\documentclass[10pt]{article}
\usepackage{amsfonts,amssymb,amsbsy,amsmath,amsthm,enumerate,verbatim}
\usepackage{mathrsfs}

\usepackage[colorlinks=true]{hyperref}
\hypersetup{urlcolor=blue,citecolor=red}

\newtheorem{theorem}{Theorem}[section]

\newtheorem{lemma}[theorem]{Lemma}
\newtheorem{definition}[theorem]{Definition}

\theoremstyle{definition}
\newtheorem{remark}[theorem]{Remark}

\newcommand{\bel}{\begin{equation} \label}
\newcommand{\ee}{\end{equation}}

\newcommand{\R}{{\mathbb R}}
\newcommand{\N}{{\mathbb N}}

\newcommand{\Z}{{\mathbb Z}}

\let\epsilon=\varepsilon
\let\phi=\varphi
\def\dans{\longrightarrow}

\def\beq{\begin{equation}}
\def\eeq{\end{equation}}
\newcommand{\bea}{\begin{eqnarray}}
\newcommand{\eea}{\end{eqnarray}}
\newcommand{\beas}{\begin{eqnarray*}}
\newcommand{\eeas}{\end{eqnarray*}}

{

\providecommand{\norm}[1]{\left\lVert#1\right\rVert}

\begin{document}

\title{Uniqueness and stability results for an \\
 inverse spectral problem in a \\ 
 periodic waveguide}

\author{O. Kavian\footnote{Laboratoire de Math\'ematiques de Versailles; UMR 8100; universit\'e de Versailles Saint-Quentin; 45 avenue des Etats Unis; 78035 Versailles cedex; France.
e-mail: {\tt kavian@math.uvsq.fr}}, 
Y. Kian\footnote{Aix-Marseille Universit\'e; CNRS, CPT UMR 7332; 13288 Marseille; and Universit\'e de Toulon; CNRS, CPT UMR 7332; 83957 La Garde; France.
e-mail: {\tt yavar.kian@univ-amu.fr}}
\& E. Soccorsi\footnote{Aix-Marseille Universit\'e; CNRS, CPT UMR 7332; 13288 Marseille; and Universit\'e de Toulon; CNRS, CPT UMR 7332; 83957 La Garde; France.
e-mail: {\tt eric.soccorsi@univ-amu.fr}}
} 

\maketitle

\begin{abstract}
\noindent Let $\Omega =\omega\times{\Bbb R}$ where $\omega\subset {\Bbb R}^2$ be a bounded domain, and $V : \Omega \dans {\Bbb R}$ a bounded potential which is $2\pi$-periodic in the variable $x_{3}\in {\Bbb R}$. We study the inverse problem consisting in the determination of $V$, through the boundary spectral data of the operator $u\mapsto Au := -\Delta u + Vu$, acting on $L^2(\omega\times(0,2\pi))$, with quasi-periodic and Dirichlet boundary conditions. More precisely we show that if for two potentials $V_{1}$ and $V_{2}$ we denote by $(\lambda_{1,k})_{k}$ and $(\lambda_{2,k})_{k}$ the eigenvalues associated to the operators $A_{1}$ and $A_{2}$ (that is the operator $A$ with $V := V_{1}$ or $V:=V_{2}$), then if $\lambda_{1,k} - \lambda_{2,k} \to 0$ as $k \to \infty$ we have that $V_{1} \equiv V_{2}$, provided one knows also that $\sum_{k\geq 1}\|\psi_{1,k} - \psi_{2,k}\|_{L^2(\partial\omega\times[0,2\pi])}^2 < \infty$, where $\psi_{m,k} := \partial\phi_{m,k}/\partial{\bf n}$. We establish also an optimal Lipschitz stability estimate. The arguments developed here may be applied to other spectral inverse problems, and similar results can be obtained.

\end{abstract}

\noindent{\bf Keywords: }{
Inverse spectral problem, stability, uniqueness, periodic waveguide, Schr\"odinger operator, periodic potential\\

}


\section{Introduction}
\label{sec-intro}
\setcounter{equation}{0}

In the present paper we study two related inverse spectral problems in which a potential is identified through an incomplete boundary spectral data.
\medskip

Let $\omega \subset {\Bbb R}^2$ be a bounded   domain. On the one hand set \begin{equation}\label{eq:Def-Y-Gamma}
Y := \omega\times(0,2\pi)\qquad\mbox{and}\qquad
\Gamma := \partial\omega\times[0,2\pi],
\end{equation} 
and on the other hand consider an infinite waveguide $\Omega$ with
\begin{equation}\label{eq:Def-Omega}
\Omega :=\omega \times \R \qquad\mbox{and}\qquad
\partial\Omega = \partial\omega\times {\Bbb R}.
\end{equation} 
We may assume, without loss of generality, that the cross section $\omega$ of the waveguide contains the origin $0_{\R^2}$ of $\R^2$. For simplicity we assume that $\omega $ is  a $\mathcal C^{1,1}$  domain. Nevertheless, with some additional arguments most of the results of this paper (Theorems \ref{thm:Uniqueness-QP}--\ref{thm:Stability}) can be treated when $\omega$ is assumed to be only a Lipschitz domain.
For the sake of brevity of notations, we write $x=(x',x_3)$ with $x'=(x_1,x_2) \in \omega$ for every $x=(x_1,x_2,x_3) \in \Omega$. 
\medskip


The main problem we study, and whose solution is a consequence of a result presented a few lines below, concerns an inverse spectral problem in a waveguide given by $\Omega$. We consider a real valued electric potential $V \in L^\infty(\Omega;{\Bbb R})$, which is $2 \pi$-periodic with respect to the infinite variable $x_3$. Namely, we assume that $V$ satisfies
\bel{eq:V-per}
V(x',x_3+2\pi)=V(x',x_3),\quad \forall\, x_3 \in \R,
\ee
and then we define the self-adjoint operator $(A,D(A))$ acting in $L^2(\Omega)$ by
\begin{equation}\label{eq:Def-A-1}
Au :=-\Delta u + Vu, \quad\mbox{for }\; u \in D(A)
\end{equation}
with its domain
\begin{equation}\label{eq:Def-Dom-A-1}
D(A) :=\left\{u \in H^1_0(\Omega) \; ; \; -\Delta u + Vu \in L^2(\Omega) \right\}.
\end{equation}
We are interested in the problem of determining $V$ from the partial knowledge of the spectral data associated with $A$. However, the operator $(A,D(A))$ being self-adjoint and its resolvent not being compact, it may have a continuous spectrum contained in an interval of type $[\lambda_{*},+\infty)$: thus in the first place one should state precisely what is meant by an inverse spectral problem. To make this statement more precise, we are going to recall the definition of the (full) spectral data associated with the operator $A$, but before doing so we state another result closely related to the above problem.

This result concerns the following inverse spectral problem: let $Y$ be as in \eqref{eq:Def-Y-Gamma} and  consider a real valued potential $V \in L^\infty(Y)$ and, for a given fixed $\theta \in [0,2\pi)$, let $(\lambda_{k}(\theta),\phi_{\theta,k})_{k \geq 1}$ be the eigenvalues and normalized eigenfunctions of the realization of the operator $-\Delta + V$ with quasi-periodic and Dirichlet boundary conditions, more precisely those eigenvalues and eigenfunctions given by
\begin{equation}\label{eq:Eigen-Pb-0}
\left\{
\begin{array}{rcll}
-\Delta \phi_{\theta,k} + V\phi_{\theta,k} & = & \lambda_{k}(\theta)\phi_{\theta,k} & \mbox{in }\, Y,\\
\phi_{\theta,k}(\sigma) & = & 0,& \sigma\in \Gamma,\\
\phi_{\theta,k}(x',2\pi) & = & {\rm e}^{{\rm i}\theta} \phi_{\theta,k}(x',0) , & x'\in\omega,\\
\partial_{3} \phi_{\theta,k}(x',2\pi) & = & {\rm e}^{{\rm i}\theta}\partial_{3} \phi_{\theta,k}(x',0) , & x'\in\omega.
\end{array}
\right.
\end{equation}
Then we show that if $N \geq 1$ is a given integer, knowledge of 
$$\lambda_{j}(\theta), \; 1_{\Gamma}{\partial\phi_{\theta,j} \over \partial{\bf n}}
\quad\mbox{for }\, j \geq N + 1,$$
allows us to identify the potential $V$ in $Y$. The novelty of this result, in contrast with analogous results in the literature (cf. mainly A.Nachman, J.~Sylvester \& G.~Uhlmann \cite{NSU}, H.~Isozaki \cite{Isozaki}), is that the normal derivatives  of the eigenfunctions, $\partial\phi_{\theta,j}/\partial{\bf n}$, are assumed to be known only on the part $\Gamma = \partial\omega\times[0,2\pi]$ of the boundary $\partial Y$ of the domain $Y$. More precisely we show the following:

\begin{theorem}
\label{thm:Uniqueness-QP} 
Let $Y$ and $\Gamma$ be as in \eqref{eq:Def-Y-Gamma}, $\theta \in [0,2\pi)$ and, for $m = 1, 2$, let $V_m \in L^\infty(Y;{\Bbb R})$. We denote by $(\lambda_{m,k}(\theta),\phi_{m,\theta,k})_{k \geq 1}$ the eigenvalues and normalized eigenfunctions given by the eigenvalue problem \eqref{eq:Eigen-Pb-0} where $V := V_{m}$, for $m = 1$ or $m = 2$, and denote $\psi_{m,\theta,k} := \partial\phi_{m,\theta,k}/\partial{\bf n}$.
Assume that for an integer $N \geq 1$ we have
\begin{equation}\label{eq:Cnd-BSD-0} 
\forall \, k \geq N + 1,\quad \lambda_{1,k}(\theta) = \lambda_{2,k}(\theta),
\quad\mbox{and}\quad
\psi_{1,\theta,k} = \psi_{2,\theta,k}
\quad\mbox{on }\, \Gamma.
\end{equation}
Then we have $V_1 \equiv V_2$.
\end{theorem}

As we shall see below in Theorem \ref{thm:Uniqueness-Asymptotic}, in order to conclude that $V_{1} = V_{2}$, it is suffient to have 
\begin{equation}\label{eq:Ecart-lambda-psi}
\lim_{k\to\infty}(\lambda_{1,k}(\theta) - \lambda_{2,,k}(\theta)) = 0 \quad\mbox{and}\quad
\sum_{k\geq 1}\|\psi_{1,\theta,k} - \psi_{2,\theta,k}\|_{L^2(\Gamma)}^2 < \infty,
\end{equation}
which is a much weaker condition than \eqref{eq:Cnd-BSD-0}.

In order to explain and state our main result concerning waveguides, in the next subsection we recall what is meant by boundary spectral data for an unbounded domain such as $\Omega = \omega\times{\Bbb R}$.


\subsection{The spectral data of the operator $A$}\label{s-sec:spectraldata}
In order to make a precise statement about the structure of the spectral data of the operator $A$, we begin with a simple example in which an operator $L$ and its spectrum ${\rm sp}(L)$ are described in terms of a family of operators $(L_{\theta})_{\theta\in[0,2\pi)}$ and their spectrums $({\rm sp}(L_{\theta}))_{\theta\in[0,2\pi)}$.

One of the main tools in the analysis of operators with periodic coefficients, such as the one given by $A$, is the Floquet-Bloch-Gel'fand transform
$U$, defined for every $f \in \mathscr{S}(\R)$, by
$$
\mbox{for }\; x,\, \theta \, \in [0,2\pi), \qquad (U f)_\theta (x) := \sum_{k=-\infty}^{+\infty}{\rm e}^{-{\rm i}k\theta}f(x+2k\pi).
$$
In view of M.~Reed \& B.~Simon \cite[\S XIII.16]{RS4}, the above operator $U$ can be extended to a unitary operator from $L^2(\R)$ onto the Hilbert space defined as the direct integral sum
$$\int_{(0,2\pi)}^{\oplus}L^2(0,2\pi) \, \frac{d \theta}{2\pi}.$$ 
Now consider the operator $L$ defined by $Lf := -\partial_{xx}f$, the one dimensional Laplacian on $L^2({\Bbb R})$, with its domain $D(L) := H^2({\Bbb R})$. One can easily check that
\bel{1D}
U L U^{-1} = \int_{(0,2\pi)}^{\oplus} L_{\theta} \, \frac{d \theta}{2\pi},
\ee
where, for a fixed $\theta \in [0,2\pi)$, the operator $L_{\theta}$ is defined by 
$$L_{\theta} (U f)_{\theta} := -\left(U \partial_{xx}f \right)_{\theta} $$
for any $f \in H^2(\R)$. However, since 
$$(U f)_{\theta}(2 \pi) = {\rm e}^{{\rm i}\theta} (U f)_{\theta}(0),
\quad\mbox{and}\quad 
(U f)_{\theta}'(2 \pi)= {\rm e}^{{\rm i}\theta} (U f)_{\theta}'(0),$$
the operator $L_{\theta}$ acts in $L^2(0,2\pi)$ as the operator $\phi \mapsto -\partial_{xx}\phi$ with quasiperiodic boundary conditions: this means that the domain of $L_{\theta}$ is given by
$$D(L_{\theta}) := \left\{ \phi \in H^2(0,2\pi) \; ; \; \phi(2\pi) - {\rm e}^{{\rm i}\theta}\phi(0) = \phi'(2\pi) - {\rm e}^{{\rm i}\theta}\phi'(0) = 0\right\}.$$ 
Thus $L_{\theta}\phi = - \partial_{xx}\phi$ for $\phi \in D(L_{\theta})$, and the relation \eqref{1D} gives a decomposition of $L$ in terms of the direct integral sum of the operators $L_{\theta}$. Indeed one can express also the spectrum of ${\rm sp}(L)$ in terms of the spectrums of the operators ${\rm sp}(L_{\theta})$, which is precisely the sequence 
$${\rm sp}(L_{\theta}) = \left\{ \lambda_{j}(\theta) \; ; \; j \in {\Bbb Z} \right\},\quad\mbox{with }\, 
\lambda_{j}(\theta) := \left(j + {\theta \over 2\pi}\right)^2,$$
and, as a matter of fact, noting that $\lambda_{j}([0,2\pi)) = [j^2,(j + 1)^2)$ for $j \geq 0$, we have
$${\rm sp}(L) =  \overline{\bigcup_{j \in {\Bbb Z}}\lambda_{j}([0,2\pi))} = [0,\infty).$$

The same procedure can be applied to the operator $A$ given by \eqref{eq:Def-A-1}, albeit with a slight modification: indeed, since we are dealing with functions of $x \in \Omega := \omega\times {\Bbb R}$ in this framework, we shall rather consider the transform defined for any $f \in \mathscr{S}(\overline{\omega}\times{\Bbb R})$ by
$$\mbox{for } x' \in \omega,\;\mbox{and}\; x_3, \theta \in [0,2\pi),\qquad
(U f)_\theta (x',x_3)=\sum_{k=-\infty}^{+\infty} {\rm e}^{-{\rm i}k\theta} f(x',x_3+2k\pi),
$$
and then suitably extend it to a unitary operator from $L^2(\Omega)$ onto the direct integral sum
$$\int_{(0,2\pi)}^{\oplus}L^2(Y) \,\frac{d \theta}{2\pi},$$
where we recall that $Y$ is defined in \eqref{eq:Def-Y-Gamma}. 
The potential $V$ being periodic with respect to $x_3$ by the assumption \eqref{eq:V-per}, we obtain in a similar way to \eqref{1D}, that
\bel{decomposition}
U A U^{-1}=\int_{(0,2\pi)}^{\oplus} A_{\theta}  \frac{d \theta}{2\pi}.
\ee
Here, for each fixed $\theta\in [0,2\pi)$ the operator $A_{\theta}$ acts in $L^2(Y)$ as $-\Delta+V$ on its domain, composed of those functions $\phi \in H^2(Y)$ such that  
\begin{equation}\label{eq:CndDirichlet}
\forall\,\sigma'\in \partial\omega,\quad \forall\, x_{3}\in (0,2\pi),\qquad
\phi(\sigma',x_{3}) = 0,
\end{equation}
and 
\begin{equation}\label{eq:CndQP}
\phi(\cdot,2\pi) - {\rm e}^{{\rm i}\theta}  \phi(\cdot,0) = \partial_{3} \phi(\cdot,2\pi) 
- {\rm e}^{{\rm i}\theta}\partial_{3} \phi(\cdot,0) 
= 0\quad 
\mbox{in }\,\omega.
\end{equation}
Thus the operator $A_{\theta}$ is given by 
$$
A_{\theta}\phi := - \Delta\phi + V\phi,
$$
for $\phi \in D(A_{\theta})$ defined to be
$$
D(A_{\theta})  := \left\{\phi \in H^1(Y) \; ; \;A_{\theta}\phi \in L^2(Y), \; \phi\;\mbox{satisfies \eqref{eq:CndDirichlet} and \eqref{eq:CndQP}}\right\}. 
$$
It is clear that for each $\theta \in [0,2\pi)$ the operator $A_{\theta}$ is self-adjoint, and that the imbedding of $D(A_{\theta})$ (endowed with its graph norm) into $L^2(Y)$ is compact: this means that $A_{\theta}$ has a compact resolvent and thus its spectrum  is composed of a sequence of real numbers 
$\{\lambda_{k}(\theta)\; ; \; k \in \N^* \}$, where  these numbers are assumed to be ordered in a non-decreasing order (repeated according to their multiplicity), and $\lambda_{k}(\theta) \to + \infty$ as $k \to +\infty$. Actually the spectrum of $A$ is determined in terms of the spectrums of $(A_{\theta})_{\theta\in [0,2\pi)}$, by the relation:  
\bel{spec}
{\rm sp}(A)=\overline{\bigcup_{k \in \N^*} \lambda_{k}([0,2\pi))}.
\ee
Moreover, the spectrum of $A$ is purely absolutely continuous (cf. N.~Filonov \& I.~Ka\-chkovskii \cite[Theorem 2.1]{FK}), which amounts to say that the so-called band functions $\theta \mapsto \lambda_j(\theta)$, $j \in \N^*$, are non constant.
\medskip

To go further and say a few words about the {\it generalized eigenfunctions\/} of $A$, we introduce a family $\{ \phi_{\theta,k}\; ; \; k \in \N^* \}$ of eigenfunctions of the operator $A_{\theta}$, which satisfy
$$ A_{\theta} \phi_{\theta,k} = \lambda_{k}(\theta) \phi_{\theta,k}\quad\mbox{in }\, Y, $$
and form an orthonormal basis of $L^2(Y)$. For $k \in \N^*$ and $\theta \in [0,2 \pi)$, we define a function $\Phi_{\theta,k}$ by setting
\bel{def:Phi-k} 
\mbox{for }\, x = (x',x_3) \in Y,\quad n \in\Z, \qquad \Phi_{\theta,k}(x',x_3+2n \pi) := {\rm e}^{{\rm i} n \theta} \phi_{\theta,k}(x),
\ee
so that for any $\chi \in C^\infty_{c}({\Bbb R})$, the function $x\mapsto \chi(x_{3})\,\Phi_{\theta,k}(x)$ belongs to the domain $D(A)$. 
For any $k \in \N^*$ and $\theta \in [0,2\pi)$, it is easy to check that
$$(-\Delta + V) \Phi_{\theta,k} = \lambda_{k}(\theta) \Phi_{\theta,k}\qquad\mbox{in }\, \Omega,$$ 
in the sens of the distributions. Therefore, for any given $k \in \N^*$ and $\theta \in [0,2\pi)$, each $\Phi_{\theta,k}$ is a  generalized eigenfunction of $A$ associated with the generalized eigenvalue $\lambda_{k}(\theta)$. 
Furthermore, the family $\left\{ \Phi_{\theta,k} \; ; \; k \in \N^*,\ \theta \in [0,2 \pi) \right\}$ is a complete system of generalized eigenfunctions of $A$, in the sense that upon setting
$$u_{k}(\theta) := \int_{Y} u(x',x_3) \overline{\Phi_{\theta,k}(x',x_3)} dx' dx_3, $$
for $u\in L^2(\Omega)$, the mapping
$$ u \mapsto \left\{ u_{k}(\theta)\; ; \; k \in \N^*,\; \theta \in [0,2 \pi) \right\}, $$
defines a unitary operator from $L^2(\Omega)$ onto $\bigoplus_{k \in \N^*} L^2(0,2 \pi)$, that is for any $u,v \in L^2(\Omega)$ we have:
$$(u|v)_{L^2(\Omega)} = \sum_{k \geq 1}\int_{0}^{2\pi} u_{k}(\theta)\overline{v_{k}(\theta)}\, {d\theta \over 2\pi}\, .$$
Now, the (full) Floquet spectral data associated with the operator $A$ is defined as
$$\left\{ (\lambda_{k}(\theta), {\rm span}(\Phi_{\theta,k})) \; ; \; k \in \N^*,\; \theta \in [0,2\pi) \right\}.$$
Often, with two abuses of notations, we shall denote the above full Floquet spectral data set as
\begin{equation}\label{eq:Def-FSD}
{\rm FSD}(V) := \left\{ (\lambda_{k}(\theta), \phi_{\theta,k}) \; ; \; k \in \N^*,\; \theta \in [0,2\pi) \right\},
\end{equation}
that is in the first place we use the eigenfunctions $\phi_{\theta,k}$ defined on $\omega\times(0,2\pi)$ instead of $\Phi_{\theta,k}$: clearly this does not create any ambiguity since $\Phi_{\theta,k}$ is known in a unique manner through the definition \eqref{def:Phi-k}. The next abuse of notations is owed to the fact that we omit to say that what is indeed important is the eigenspace ${\rm span}(\Phi_{\theta,k})$, or ${\rm span}(\phi_{\theta,k})$, rather than each eigenfunction $\Phi_{\theta,k}$ or $\phi_{\theta,k}$, in particular when the Floquet eigenvalue $\lambda_{k}(\theta)$ is a multiple eigenvalue. 

Moreover, in accordance with G. Eskin, J. Ralston \& E. Trubowitz \cite[\S I.6]{ERT}, for any $\theta\in[0,2\pi)$ fixed, the set $\{(\lambda_j(\theta),\Phi_{\theta,j}) \; ; \; j \in \N^*\}$ will be referred to as the {\it Floquet spectral data\/} (or equivalently, the {\it Floquet eigenpairs\/}) associated with the operator $A$ at $\theta\in [0,2\pi)$.

The Floquet boundary spectral data associated to the potential $V$ will be the set of eigenpairs $(\lambda_{k}(\theta),{\rm span}(\psi_{\theta,k}))$, where
\begin{equation}\label{eq:Def-psi-k}
\psi_{\theta,k} := {\partial\phi_{\theta,k} \over \partial{\bf n}}\, ,
\end{equation}
and for a fixed $k$ and $\theta$, the finite dimensional space ${\rm span}(\psi_{\theta,k})$ is meant to be spanned by the normal derivatives of all eigenfunctions associated to $\lambda_{k}(\theta)$. The Floquet boundary spectral data set will be denoted, again by an abuse of notations, as
\begin{equation}\label{eq:Def-BFSD}
\begin{cases}
&{\rm FBSD}(V) := \bigcup_{\theta\in [0,2\pi]} {\rm FBSD}(\theta,V), \cr
&\mbox{where }\,
{\rm FBSD}(\theta,V) := \left\{ (\lambda_{k}(\theta), \psi_{\theta,k}) \; ; \; k \in \N^*\right\}\cr
\end{cases}
\end{equation}


\subsection{Main results in an infinite waveguide}
\label{sec-mainresults}
We consider two potentials $V_m \in L^\infty(\Omega;\R)$, $m=1,2$, that are $2 \pi$-periodic with respect to  $x_3$,
\bel{eq:V-per2}
V_m(x',x_3+2\pi) = V_m(x',x_3),\qquad x' \in \omega,\ x_3 \in \R,
\ee
and we call $A_m$ (resp. $A_{m,\theta}$ for $\theta \in [0,2\pi)$) the operator obtained by substituting $V_m$ for $V$ in the definition of the operator $A$ (resp. $A_{\theta}$), so that we have:
\bel{decomposition2}
U A_m U^{-1}=\int_{(0,2\pi)}^{\oplus} A_{m,\theta} \,  \frac{d \theta}{2\pi},\qquad\mbox{for }\, m=1,2.
\ee
Further we note $\left\{(\lambda_{m,k}(\theta),\phi_{m,\theta,k}) \; ; \; k \in \N^*,\; \theta \in [0,2 \pi) \right\}$ the full spectral data associated with $A_m$, for $m=1,2$, as defined in \eqref{eq:Def-FSD}. 

\medskip

The following is an identification result for a potential defined in a waveguide:

\begin{theorem}
\label{thm:Uniqueness} 
For $m = 1, 2$ let $V_m \in L^\infty(\Omega;\R)$ satisfy \eqref{eq:V-per2}, and denote $\psi_{m,k,\theta}$ as in \eqref{eq:Def-psi-k}. Assume that for some $\theta_{0}\in [0,2\pi)$ and some integer $N \geq 1$ we have
\bel{t1a} 
\forall \,  k \geq N + 1,\quad \lambda_{1,k}(\theta_0) = \lambda_{2,k}(\theta_0)\quad\mbox{and}\quad
\psi_{1,\theta_{0},k} = \psi_{2,\theta_{0},k}\quad\mbox{on }\, \Gamma.
\ee
Then we have $V_1 \equiv V_2$.
\end{theorem}

Theorem \ref{thm:Uniqueness} yields that the knowledge of the Floquet boundary spectral data, with the possible exception of finitely many generalized eigenpairs, at one arbitrary $\theta_0 \in [0,2 \pi)$, uniquely determines the operator $A$. The claim seems quite surprising at first sight, since the full boundary spectral data of $A$ is the collection of the Floquet data at $\theta$, for $\theta$ evolving in $[0,2 \pi)$. Nevertheless, we point out that this result is in accordance with G. Eskin, J. Ralston \& E. Trubowitz \cite[Theorem 6.2]{ERT}, where Floquet isospectrality at $\theta=0$ for Schr\"odinger operators with analytic periodic potential in $\R^n$, $n \geq 1$, implies Floquet isospectrality for all $\theta \in [0,2\pi)$. 

\medskip

As a matter of fact, we show the stability result stated in Theorem \ref{thm:Stability} below, which yields a much stronger uniqueness result. Indeed, notwithstanding the fact that the main interest of Theorem \ref{thm:Uniqueness} lies in its simplicity, notice that under the assumptions \eqref{t1a} one has also 
\begin{equation*}\label{eq:Cnd-Ecart-ell-2}  
\sum_{k=1}^{\infty} \, \|\psi_{1,\theta_{0},k} - \psi_{2,\theta_{0},k}\|^2_{L^2(\Gamma)} < \infty.
\end{equation*}
Actually, the above condition is sufficient to state a stability result in terms  of the asymptotic distance between the eigenvalues $|\lambda_{1,k}(\theta_{0}) - \lambda_{2,k}(\theta_{0})|$, as stated in the following:

\begin{theorem}
\label{thm:Stability} 
Let $M >0$ be such that for $V_m \in L^\infty(\Omega;\R)$ fulfilling \eqref{eq:V-per2}, and $m = 1, 2$, 
we have $\|V_{m}\|_{\infty} \leq M$, and denote $V = (V_1-V_2)1_{Y}$.  Let $\widehat{V}$ be the Fourier transform of $V$ defined by
$$\widehat{V}(\xi',j) := (2\pi)^{-3/2}\int_0^{2\pi}\int_{\R^2}V(x',x_3)e^{-{\rm i}(\xi'\cdot x' + j x_3)}dx'dx_3,$$
for $\xi'\in{\Bbb R}^2$ and $j\in{\Bbb Z}$.
For some given $\theta_{0}\in [0,2\pi)$ assume that
\begin{equation}\label{eq:Ecart-ell-2-psi} 
\sum_{k=1}^\infty \norm{\psi_{1,\theta_{0},k} - \psi_{2,\theta_{0},k}}_{L^2(\Gamma)}^2 < \infty.
\end{equation}
Then there exists a positive constant $c$ depending only on $\omega$ and $M$ such that for $(\xi',j)\in\R^2\times \mathbb{Z}$, and all $N \geq 1$, the following stability estimate holds:
\begin{equation}\label{eq:Estim-stability}
|\widehat{V}(\xi',j)| \leq c\, 
\sup_{k\geq N}|\lambda_{1,k}(\theta_0)-\lambda_{2,k}(\theta_0)|.
\end{equation}
\end{theorem}

\noindent Since one can easily see that in general one has 
$$|\lambda_{1,k}(\theta_{0}) - \lambda_{2,k}(\theta_{0})| \leq \|V_{1} - V_{2}\|_{\infty} = \|V\|_{\infty},$$ the above stability estimate is, in some loose sense, optimal. 

It is noteworthy that \eqref{eq:Estim-stability} involves only the asymptotic distance between the eigenvalues, and does not involve explicitely any quantitative information about $\|\psi_{1,\theta_{0},k} - \psi_{2,\theta_{0},k}\|_{L^2(\Gamma)}$. This seems somewhat surprising, since one can exhibit distinct {\it isospectral\/} potentials $V_{1},V_{2}$ on certain domains $\omega$ (or equivalently domains $Y$ and potentials $V_{1},V_{2}$), such that $\lambda_{1,k}(\theta_{0}) = \lambda_{2,k}(\theta_{0})$ for all $k \geq 1$: indeed for such potentials one has $\sum_{k=1}^\infty \norm{\psi_{1,\theta_{0},k} - \psi_{2,\theta_{0},k}}_{L^2(\Gamma)}^2 = \infty$. 

Actually, from the estimate \eqref{eq:Estim-stability} one can deduce  estimates of $\norm{V}_{H^{-1}(\Omega)}$ in terms of $\delta_{0} := \sup_{k\geq N}|\lambda_{1,k}(\theta_0)-\lambda_{2,k}(\theta_0)|$. Also, with some additional assumptions, one can get estimates of $\norm{V_{1} - V_{2}}$ in other spaces such as $L^2(\Omega)$, $L^\infty(\Omega)$,\dots. Here, in order to preserve some generality, we do not consider such further estimates directly in terms of $V_{1} - V_{2}$. 

The approach developed in the subsequent sections, to prove the stability result, allows us to tackle other inverse spectral problems for operators of the type $u \mapsto -\Delta u + Vu$ with boundary conditions such as Dirichlet, or Neumann or Fourier type boundary conditions (that is $\partial u/\partial{\bf n} + K(\sigma)u(\sigma) = 0$ on the boundary). Elsewhere we shall develop on such problems, or actually more general spectral problems such as $ -{\rm div}(a(x)\nabla \phi_{k}) + V(x)\phi_{k} = \lambda_{k}\rho(x)\phi_{k}$ with various types of boundary conditions (here, for some $\epsilon_{0} >0$ it is assumed that $\min(a(x),\rho(x)) \geq \epsilon_{0} >0$).

We would like to mention that, in the context of Dirichlet boundary conditions for the operator $u \mapsto -\Delta u + Vu$ on a bounded domain $\Omega \subset {\Bbb R}^n$, analogous estimates have been established  by M. Choulli \& P. Stefanov \cite{Choulli-Stefanov}. More precisely, 
these authors prove in the latter reference that
if the quantities $\delta_{0}$ and $\delta_{1}$ defined by 
$$\delta_{0} := \sup_{k \geq 1}k^\alpha|\lambda_{1,k} - \lambda_{2,k}|,
\qquad 
\delta_{1} := \sup_{k\geq 1} k^\beta\|\psi_{1,k} - \psi_{2,k}\|_{L^2(\Gamma)},$$
where $\alpha > 1$ and $\beta > 1 - (1/2n)$, are finite, then one has $V_{1} \equiv V_{2}$. 
Note that whenever $\delta_{0} + \delta_{1} < \infty$, then the conditions \eqref{eq:Ecart-ell-2-psi} and \eqref{eq:Cnd-Ecart-nul} of this paper are satisfied.
The approach we take in this paper follows somewhat the cited paper by M. Choulli \& P. Stefanov in the use of H.~Isozaki's representation formula (see Theorem \ref{thm:Isozaki}), however, in estimating $V_{1} - V_{2}$, in this paper we get rid of the dependence on quantities involving $\|\psi_{1,\theta_{0},k} - \psi_{2,\theta_{0},k}\|_{L^2(\Gamma)}$.
\medskip

It is clear that the following is a consequence of Theorem \ref{thm:Stability}:
\begin{theorem}
\label{thm:Uniqueness-Asymptotic} 
Let $V_m$, $m=1,2$, satisfy the same assumptions as in Theorem \ref{thm:Uniqueness}. Suppose that the following condition
\begin{equation}\label{eq:Cnd-Ecart-nul} 
\lim_{k\to\infty}|\lambda_{1,k}(\theta_{0}) -\lambda_{2,k}(\theta_{0})| = 0,
\end{equation}
and \eqref{eq:Ecart-ell-2-psi} hold  for some $\theta_0 \in [0, 2 \pi)$. Then we have $V_1 \equiv V_2$.

Analogously, if $V_{m}\in L^\infty(Y)$, with the notations and assumptions of Theorem \ref{thm:Uniqueness-QP}, when  \eqref{eq:Ecart-lambda-psi} is satisfied, then we have $V_{1} \equiv V_{2}$.
\end{theorem}

Otherwise stated, it is enough that the Floquet boundary spectral data associated with $V_1$ and $V_2$ for a single arbitrary $\theta_0 \in [0, 2 \pi)$, coincide asymptotically in the sense of \eqref{eq:Ecart-ell-2-psi} and \eqref{eq:Cnd-Ecart-nul}, for the uniqueness result of Theorems \ref{thm:Uniqueness-QP} and \ref{thm:Uniqueness} to hold. 

\medskip


\subsection{Inverse spectral theory: a short review of the existing literature}

The study of inverse spectral problems goes back at least to V.A.~Ambarzumian \cite{Ambarzumian} who investigated in 1929 the inverse spectral problem of determining the real potential $V$ appearing in the Sturm--Liouville operator $A=-\partial_{xx} + V$, acting in $L^2(0,2\pi)$, from partial spectral data of $A$. He proved in \cite{Ambarzumian} that $V=0$ if and only if the spectrum of the periodic realization of $A$ equals $\{ k^2\; ; \; k \in \N \}$. For the same operator acting on $L^2(0,\pi)$, but endowed with homogeneous Dirichlet boundary conditions, G.~Borg \cite{Borg-1946} and N.~Levinson \cite{Levinson-1949} established that while the Dirichlet spectrum $\{ \lambda_{k} \; ; \; k \in \N^* \}$ does not uniquely determine $V$, nevertheless assuming that $\varphi_{k}'(0) = 1$ for $k \geq 1$, additional spectral data, namely $\{ \| \phi_{k} \|_{L^2(0,\pi)} \; ; \; k \in \N^* \}$ is needed, where $\{ \varphi_{k}\; ; \; k \in \N^* \}$ is an $L^2(0,\pi)$-orthogonal basis of eigenfunctions of $A$. I.M.~Gel'fand and B.M.~Levitan \cite{Gelfand-Levitan-1951} proved that uniqueness is still valid upon substituting $\phi_{k}'(\pi)$ for $\| \phi_{k} \|_{L^2(0,\pi)}$ in the one-dimensional Borg and Levinson theorem.
\medskip

In 1988, the case where $\Omega$ is a bounded domain of $\R^n$, $n\geq2$, was treated by A.~Nachman, J.~Sylvester and G.~Uhlmann \cite{NSU}, and by N.G.~Novikov \cite{Novikov}. Inspired by \cite{Gelfand-Levitan-1951}, these authors proved that the boundary spectral data $\{ (\lambda_{k} ,\partial\phi_{k}/\partial{\bf n}) \; ; \; k \in \N^* \}$ uniquely determines the Dirichlet realization of the operator $A$. This result has been improved in several ways by various authors. H.~Isozaki \cite{Isozaki} (see also M.~Choulli \cite{Choulli-Book}) extended the result of \cite{NSU} when finitely many eigenpairs remain unknown, and, recently, M.~Choulli and P.~Stefanov \cite{Choulli-Stefanov} proved uniqueness in the determination of $V$ from the asymptotic behaviour of $(\lambda_{k},\partial \phi_{k}/\partial{\bf n})$ as $k \to \infty$.
Moreover, B.~Canuto and O.~Kavian \cite{CK1, CK2} proved that in problems such as $-{\rm div}(c\nabla\phi_{k}) + V\phi_{k} = \lambda_{k}\rho\phi_{k}$, where the conductivity $c$ and the density $\rho$ satisfy $\min(\rho,c) \geq \epsilon_{0}$ for some $\epsilon_{0} > 0$, two out of the three functions $c,\rho, V$ are uniquely determined from the boundary spectral data. 
In all these results either of Dirichlet or Neumann boundary conditions can be assumed for the eigenfunctions $\phi_{k}$.
\medskip

All the above mentioned results were obtained when $\Omega$ is a bounded domain and thus the operator $A$ has a purely discrete spectrum. G.~Borg \cite{Borg-1952} and V.A.~Marchenko \cite{Marchenko-1973} independently examined the uniqueness issue in the inverse problem of determining the electric potential of $-\partial_{xx} + V$ in $\Omega=\R_+^*$, with Fourier flux boundary condition $\alpha\psi(0) - \psi'(0) = 0$ at $x=0$. They proved that when there is no continuous spectrum, two sets of discrete spectrums associated with two distinct boundary conditions at $x=0$ uniquely determine the potential and the two boundary conditions. F.~Gesztesy and B.~Simon \cite{GS, GS1, GS2, S} and T.~Aktosun and R.~Weder \cite{AW} extended the Borg-Marchenko result in presence of a continuous spectrum, where either the Krein's spectral shift function, or an appropriate set containing the discrete eigenvalues and the continuous part of the spectral measure, are used as the known data.
To the best of our knowledge, there is only one multi-dimensional Borg-Marcheko uniqueness result available in the mathematical literature, that of F.~Gesztesy and B.~Simon \cite[Theorem 2.6]{GS}, where the special case of three-dimensional Schr\"odinger operators with spherically symmetric potentials is studied.
\medskip

Finally, let us mention for the sake of completeness that the stability issue in the context of inverse spectral problems has been examined by G.~Alessandrini \& J.~Sylvester \cite{Alessandrini-Sylvester}, M. Bellassoued, M. Choulli \& M. Yamamoto \cite{BCY}, M. Bellassoued \& D. Dos Santos Ferreira \cite{BF}, M.~Choulli \cite{Choulli-Book}, M.~Choulli \& P.~Stefanov \cite{Choulli-Stefanov}, that inverse spectral problems stated on Riemannian manifolds have been investigated in M. Bellassoued \& D. Dos Santos Ferreira \cite{BF}, and in Y. Kurylev, M. Lassas \& R. Weder \cite{BF, KLW}, and that isospectral sets of Schr\"odinger operators with periodic potentials or Schr\"odinger operators defined on a torus, have been characterized in G.~Eskin \cite{Eskin-1989}, G. Eskin, J. Ralston \& E. Trubowitz \cite{ERT}, V.~Guillemin \cite{Guillemin}.
\medskip

We should point out that the problem under examination in this paper is a three-dimensional Borg-Levinson inverse problem, stated on the infinitely extended cylindrical domain $\Omega=\omega \times \R$, associated with an operator $A=-\Delta+V$ 
of (as already mentioned in Subsection \ref{s-sec:spectraldata}) purely absolutely continuous spectral type. As far as we know, Theorem \ref{thm:Uniqueness} is the only multi-dimensional Borg-Levinson uniqueness result for an operator with continuous spectrum. The method used here can be applied to obtain analogous results when $\Omega = \omega \times {\Bbb R}$ and $\omega \subset {\Bbb R}^{d-1}$ is a $C^{1,1}$ bounded domain, provided $d \geq 3$. However, in the present paper we do not develop in that direction.

\bigskip
The remainder of this paper is organized as follows.
In Section \ref{sec:Prelim} we study the boundary value problem with non-homogeneous boundary data, associated with $A_{\theta}$ for $\theta\in[0,2\pi)$. The notations used throughout, as well as functional spaces needed in the analysis of our problem are presented there. In Section \ref{sec:Isozaki} we give a representation formula for the Poincar\'e--Steklov operators $\Lambda_{\theta,V-\lambda}$, and in Section \ref{sec-mainresults} we present the proof of Theorems \ref{thm:Uniqueness-QP}--\ref{thm:Uniqueness-Asymptotic}, by using the results of Section \ref{sec:Prelim} and Section \ref{sec:Isozaki}. 


\section{Notations and preliminary results}
\label{sec:Prelim}

We assume that $\omega \subset {\Bbb R}^2$ is a $C^{1,1}$ bounded domain and that $\Omega = \omega \times{\Bbb R}$, the generic point $x \in \Omega$ being denoted $x=(x',x_{3})$ with $x' \in \omega$ and $x_{3} \in {\Bbb R}$. The Laplacian $\Delta$ is decomposed into $\Delta = \Delta' + \partial_{33}$, with the conventions $\partial_{jj} := \partial^2/\partial x_{j}^2$ and $\Delta' := \partial_{11} + \partial_{22}$; analogously we write the gradient $\nabla = (\nabla',\partial_{3})$ with $\nabla' := (\partial_{1},\partial_{2})$. The real valued potential $V \in L^\infty(\Omega)$ satisfies the periodicity condition \eqref{eq:V-per}. Recall that the domain $Y := \omega\times(0,2\pi)$ is defined in \eqref{eq:Def-Y-Gamma}, and represents a cell whose infinite reproduction yields $\Omega = \omega \times {\Bbb R}$. The scalar product of $L^2(Y)$ is denoted by $(u|v)$ for $u,v\in L^2(Y)$, and its associated norm by $\|\cdot\|$.

We denote by $\langle f,\psi\rangle$ the duality between $\psi \in H^{1/2}(\Gamma)$ and $f$ belonging to the dual of $H^{1/2}(\Gamma)$. However, when in $\langle f,\psi \rangle $ both $f$ and $\psi$ belong to $L^2(\Gamma)$, to make things simpler $\langle \cdot,\cdot\rangle$ can be interpreted as the scalar product of $L^2(\Gamma)$, namely
$$\langle f,\psi\rangle = \int_{\Gamma}\psi(\sigma)\, \overline{f(\sigma)}\,d\sigma.$$
\medskip

It is well known that the trace operator $\gamma_{0} : C^1(\overline{Y}) \longrightarrow C(\partial Y)$ defined by $\gamma_{0}(\phi) := \phi_{|\partial Y}$ can be extended to $H^1(Y)$. For $\theta \in [0,2\pi)$ fixed, we denote by $H^1_{\theta}(Y)$ the closed subspace of those functions $u \in H^1(Y)$ satisfying in the sense of traces
\begin{equation}
u(x',2\pi) = {\rm e}^{{\rm i}\theta}u(x',0)\quad\mbox{for }\,x' \in \omega, \label{eq:Cnd-quasi-per}
\end{equation}
and we shall set
$$H^{1/2}_{\theta}(\partial Y) := \gamma_{0}(H^1_{\theta}(Y)).$$
The space $H^1_{0,\theta}(Y)$ denotes the closed subspace of those functions $u \in H^1_{\theta}(Y)$ satisfying in the sense of traces
\begin{equation}
u(\sigma',x_{3}) = 0\quad \mbox{for }\, (\sigma',x_{3}) \in \Gamma := \partial\omega\times[0,2\pi].
\label{eq:Cnd-Dirichlet}
\end{equation}
Since the imbeddings $H^1_{\theta}(Y) \subset L^2(Y)$ and $H^1_{0,\theta}(Y) \subset L^2(Y)$ are compact, one can see that there exists $\phi_{*,\theta} \in H^1_{0,\theta}(Y)$ such that $\|\phi_{*,\theta}\| = 1$ and
$$\int_{Y}|\nabla \phi_{*,\theta}|^2\,dx = \lambda_{*,1}(\theta) := \min\left\{\int_{Y}|\nabla \phi|^2\,dx \; ; \; \phi\in H^1_{0,\theta}(Y),\; \|\phi\|^2 = 1 \right\}.$$
This implies that $\lambda_{*,1}(\theta) > 0$ (otherwise $\phi_{*,\theta}$ would be constant, and, since it belongs to $H^1_{0,\theta}(Y)$, it has a zero trace on $\Gamma$ hence one would have $\phi_{*,\theta} \equiv 0$, which is not compatible with the condition $\|\phi_{*,\theta}\| = 1$). Therefore we have the following Poincar\'e inequality on $H^1_{0,\theta}(Y)$:
\begin{equation}\label{eq:Poinca}
\forall\, u \in H^1_{0,\theta}(Y), \qquad \lambda_{*,1}(\theta) \|u\|^2 \leq \|\nabla u\|^2.
\end{equation}

\begin{remark}\label{rem:Eigen}
As a matter of fact, $\lambda_{*,1}(\theta)$ is the first eigenvalue of the Laplacian $-\Delta_{\theta}$ on $Y$, an operator which is defined as follows: consider the domain
$$D(\Delta_{\theta}) := \left\{\psi\in H^1_{0,\theta}(Y) \; ; \; \Delta \psi \in L^2(Y), \; \mbox{and }\, \psi\, \mbox{satisfies }\, \eqref{eq:CndQP} \right\},$$
and set $\Delta_{\theta}\psi := \Delta\psi$ for $\psi \in D(\Delta_{\theta})$. Then $\lambda_{*,1}(\theta)$ is the first eigenvalue of $-\Delta_{\theta}$. Actually, upon using a separation of variables in $x' \in \omega$ and $x_{3} \in (0,2\pi)$, it is easy to see that if we denote by $(\mu_{k},\phi_{k})_{k \geq 1}$ the non decreasing sequence of eigenvalues and their corresponding normalized eigenfunctions in $L^2(\omega)$ of the operator $-\Delta'$ on $H^1_{0}(\omega)$, that is the eigenvalues of the two dimensional Laplacian on $\omega$ with homogeneous Dirichlet boundary conditions, then the eigenvalues and eigenfunctions of $-\Delta_{\theta}$ are given by the sequence
$$\phi_{*,k,j}(x) := \phi_{k}(x')\exp\left({\rm i}\left({\theta \over 2\pi} + j\right) x_{3}\right),\quad  \lambda_{*,k,j} := \mu_{k} + \left({\theta \over 2\pi} + j \right)^2, $$
for $k\in {\Bbb N}^*$ and $j \in {\Bbb Z}$.
Thus $\lambda_{*,1}(\theta)$ is the smallest of the eigenvalues $\lambda_{*,k,j}$, and one sees that (recall that $\mu_{1} > 0$)
$$\lambda_{*,1}(\theta) = \mu_{1} + {\theta^2 \over 4\pi^2},\qquad\mbox{or}\qquad
\lambda_{*,1}(\theta) = \mu_{1} + {(2\pi - \theta)^2 \over 4\pi^2},
$$
according to whether $0 \leq \theta \leq \pi$ or $\pi \leq \theta < 2\pi$.
\qed
\end{remark}
Using this observation, for any given $f \in H^{1/2}_{\theta}(\partial Y)$, by minimizing the functional $\psi \mapsto \|\nabla \psi\|^2$ on the closed affine space 
$$H_{f} := \{\psi \in H^1_{\theta}(Y) \; ; \; \gamma_{0}(\psi) = f\; \mbox{on }\, \Gamma\},$$
one sees easily that  there exists a unique $F \in H^1_{\theta}(Y)$ such that
$$\|\nabla F\|^2 = \min_{\psi \in H_{f}} \|\nabla \psi\|^2.$$
As a matter of fact one checks that $F \in H^1_{\theta}(Y)$ satisfies
\begin{equation}\label{eq:Relevement}
\left\{
\begin{array}{rcll}
-\Delta F & = & 0 & \mbox{in }\, Y,\\
F(\sigma) & = & f(\sigma),& \sigma\in \Gamma,\\
F(x',2\pi) & = & {\rm e}^{{\rm i}\theta} F(x',0) , & x'\in\omega,\\
\partial_{3} F(x',2\pi) & = & {\rm e}^{{\rm i}\theta}\partial_{3} F(x',0) , & x'\in\omega,
\end{array}
\right.
\end{equation}
and moreover, for a constant $c(\theta) > 0$ depending on $Y$,
$$\|F\|_{H^1_{\theta}(Y)} \leq c(\theta)\, \|f\|_{H^{1/2}_{\theta}(\partial Y)}.$$

\begin{remark}
We do not need a uniform estimate on $F$ in terms of $\theta$, but actually one can show that in fact one has 
$$\|F\|_{H^1_{\theta}(Y)} \leq c_{*} \, \|f\|_{H^{1/2}_{\theta}(\partial Y)},$$
for a constant $c_{*} > 0$ depending only on $\omega$. \qed
\end{remark}

The operator $(A_{\theta},D(A_{\theta}))$ being defined as
\begin{eqnarray}
& A_{\theta} u := -\Delta u + V u,\quad\mbox{for }\, u\in D(A_{\theta}),\label{eq:Def-A-theta}\\
&D(A_{\theta}) := \left\{\psi \in H^1_{0,\theta}(Y) \; ; \;A_{\theta}\psi \in L^2(Y), \; \mbox{and }\, \psi\;\mbox{satisfies \eqref{eq:CndQP}}\right\},
\label{eq:Def-D-A-theta}
\end{eqnarray}
one checks easily that for $u,v \in D(A_{\theta})$, thanks to the boundary conditions satisfied by $u$ and $v$, we have
$$(A_{\theta}u|v) = \int_{Y}\nabla u(x)\cdot\overline{\nabla v(x)}\,dx + \int_{Y}V(x)\,u(x)\overline{v(x)}\,dx = (u|A_{\theta}v).$$
This allows one to see that $(A_{\theta},D(A_{\theta}))$ is a self-adjoint operator acting on $L^2(Y)$. Thanks to the compactness of the imbedding $D(A_{\theta}) \subset L^2(Y)$, one sees that $A_{\theta}$ has a compact resolvent and that the spectrum of $A_{\theta}$ is discrete and composed of the eigenvalues denoted by ${\rm sp}(A_{\theta}) = \{\lambda_{k}(\theta)\; ; \; k \geq 1\}$. If we write $V = V^+ - V^-$, with $V^\pm := \max(0,\pm V)$, we have that
$${\rm sp}(A_{\theta}) \subset [-M, + \infty),\qquad\mbox{with }\, M := \|V^-\|_{\infty}.$$
Using regularity results (for instance combining P.~Grisvard \cite[Theorem 2.2.2.3]{Gr}   with \cite[Lemma 2.2]{CKS}), one can see that $D(A_{\theta})$ is embedded continuously into $H^2(Y)$. Therefore the eigenfunctions $(\phi_{\theta,k})_{k\geq1}$ belong to $H^2(Y)$ and we have ${\partial \phi_{\theta,k} \over \partial{\bf n}}\in H^{1/2}(Y) \subset L^2(\partial Y)$.

\begin{lemma}\label{lem:Resolution}
For any $f \in H^{1/2}_{\theta}(\partial Y)$ and $\lambda \in {\Bbb C} \setminus {\rm sp}(A_{\theta})$, there exists a unique solution $u \in H_{\theta}^1(Y)$ to the equation
\bel{eq1}
\left\{ 
\begin{array}{rcll} 
-\Delta u + Vu - \lambda u & = & 0, & \mbox{in}\ Y ,\\ 
u(\sigma) & = & f(\sigma),& \sigma\in \Gamma,\\
 u(x',2\pi) & = & {\rm e}^{{\rm i}\theta} u(x',0) , & x'\in\omega,\\
\partial_{3} u(x',2\pi) & = & {\rm e}^{{\rm i}\theta}\partial_{3} u(x',0) , & x'\in\omega,
\end{array}\right.
\ee
which can be written as 
\begin{equation}\label{eq:Sol-Series}
u_{\lambda} := u = \sum_{k \geq 1} {\alpha_{k} \over \lambda - \lambda_{k}(\theta)}\,  \phi_{\theta,k},
\end{equation}
where for convenience we set
\begin{equation}\label{eq:Def-psi-alpha}
\psi_{k} := {\partial \phi_{\theta,k} \over \partial{\bf n}}, \qquad
\mbox{and}\qquad 
\alpha_{k} := \alpha_{k}(\theta,f) := \langle \psi_{k}, f \rangle.
\end{equation}
Moreover
\begin{equation}\label{eq:Series-u-lambda} 
\|u_{\lambda}\|^2_{L^2(Y)} = \sum_{k \geq 1} {|\alpha_{k}|^2 \over |\lambda - \lambda_{k}(\theta)|^2} \, ,
\end{equation}
and
\begin{equation}\label{eq:Series-u-lambda-zero} 
\|u_{\lambda}\|^2_{L^2(Y)} \to 0 \quad\mbox{as }\, \lambda \to -\infty.
\end{equation}

\end{lemma}

\begin{proof}
The solution $u$ of equation \eqref{eq1} can be written in terms of the eigenvalues and eigenfunctions $\lambda_{k}(\theta),\phi_{\theta,k}$. Indeed, since $u \in L^2(Y)$ can be expressed in the Hilbert basis $(\phi_{\theta,k})_{k\geq 1}$ as
$$u = \sum_{k\geq 1} (u|\phi_{\theta,k})\, \phi_{\theta,k}\, ,$$
taking the scalar product of the first equation in \eqref{eq1} with $\phi_{\theta,k}$ and integrating by parts twice we obtain
$$\int_{\Gamma}f(\sigma){\partial{\overline {\phi_{\theta,k}}}(\sigma) \over \partial{\bf n}}\, d\sigma = (\lambda - \lambda_{k}(\theta))\,(u|\phi_{\theta,k}),$$
which yields the expression given by \eqref{eq:Sol-Series}. 

The fact that $\|u_{\lambda}\| \to 0$ as $\lambda \to -\infty$ is a consequence of the fact that we may fix $c_{0} > 0$ large enough so that if $\lambda$ is real and such that $\lambda \leq - c_{0}$, we have $|\lambda - \lambda_{k}(\theta)|^2 \geq 1 + |\lambda_{k}(\theta)|^2$ for all $k \geq 1$, and thus
$${|\alpha_{k}|^2 \over |\lambda - \lambda_{k}(\theta)|^2} \leq  
{|\alpha_{k}|^2 \over 1 + |\lambda_{k}(\theta)|^2 } \, ,$$
so that we may apply Lebesgue's dominated convergence to the series appearing in \eqref{eq:Series-u-lambda}, as $\lambda \to -\infty$ and deduce \eqref{eq:Series-u-lambda-zero}.
\end{proof}

It is clear that the series \eqref{eq:Sol-Series} giving $u_{\lambda}$ in terms of $\alpha_{k},\lambda_{k}(\theta)$ and $\phi_{\theta,k}$, converges only in $L^2(Y)$ and thus we cannot deduce an expression of the normal derivative $\partial u_{\lambda}/\partial{\bf n}$ in terms of $\alpha_{k}, \lambda_{k}(\theta)$ and $\psi_{k}$. To circumvent this difficulty we have the following lemma:

\begin{lemma}\label{lem:v-lambda-mu}
Let $f \in H_{\theta}^{1/2}(\partial Y)$ be fixed and for $\lambda,\mu \in {\Bbb C} \setminus {\rm sp}(A_{\theta})$ let $u_{\lambda}$ and $u_{\mu}$ be the solutions given by Lemma \ref{lem:Resolution}. If we set $v := v_{\lambda,\mu} := u_{\lambda} - u_{\mu}$, then
\begin{equation}\label{eq:v-Normal-Deriv}
{\partial v \over \partial{\bf n} } = \sum_{k} 
{(\mu - \lambda)\alpha_{k} \over (\lambda - \lambda_{k}(\theta))(\mu - \lambda_{k}(\theta))}\, \psi_{k}\, ,
\end{equation}
the convergence taking place in $H^{1/2}(\Gamma)$.
\end{lemma}

\begin{proof} Let $v_{\lambda,\mu} := u_{\lambda} - u_{\mu}$; One verifies that $v_{\lambda,\mu}$ is solution to
\begin{equation}\label{eq:v-lambda}
\left\{ 
\begin{array}{rcll} 
-\Delta v_{\lambda,\mu} + V v_{\lambda,\mu} - \lambda v_{\lambda,\mu} & = & (\lambda - \mu)u_{\mu}, & \mbox{in}\ Y ,\\ 
v_{\lambda,\mu}(\sigma) & = & 0,& \sigma\in \Gamma,\\
v_{\lambda,\mu}(x',2\pi) & = & {\rm e}^{{\rm i}\theta} v_{\lambda,\mu}(x',0) , & x'\in\omega,\\
\partial_{3} v_{\lambda,\mu}(x',2\pi) & = & {\rm e}^{{\rm i}\theta}\partial_{3} v_{\lambda,\mu}(x',0) , & x'\in\omega.
\end{array}\right.
\end{equation}
Since $(u_{\mu}|\phi_{\theta,k}) = \alpha_{k}/(\mu - \lambda_{k}(\theta))$, it follows that
$$v_{\lambda,\mu} = \sum_{k}{(\lambda - \mu) \alpha_{k}\over (\lambda_{k}(\theta) - \lambda)(\mu - \lambda_{k}(\theta)) }\, \phi_{\theta,k},$$
the convergence taking place in $D(A_{\theta})$. Since the operator $v \mapsto \partial v/\partial{\bf n}$ is continuous from $D(A_{\theta})$ into $H^{1/2}(\Gamma)$, the result of the lemma follows.
\end{proof}

The next lemma states essentially that if for $m=1$ or $m=2$ we have two potentials $V_{m}$ and $u_{m} := u_{m,\mu}$ solves
\begin{equation}\label{eq:u-m}
\left\{ 
\begin{array}{rcll} 
-\Delta u_{m} + V_{m}u_{m} - \mu u_{m} & = & 0, & \mbox{in}\ Y ,\\ 
u_{m}(\sigma) & = & f(\sigma),& \sigma\in \Gamma,\\
 u_{m}(x',2\pi) & = & {\rm e}^{{\rm i}\theta} u_{m}(x',0) , & x'\in\omega,\\
\partial_{3} u_{m}(x',2\pi) & = & {\rm e}^{{\rm i}\theta}\partial_{3} u_{m}(x',0) , & x'\in\omega,
\end{array}\right.
\end{equation}
then $u_{1,\mu}$ and $u_{2,\mu}$ are {\it close\/} as $\mu \to -\infty$: in some sense the influence of the potentials $V_{m}$ is dimmed when $\mu \to -\infty$. More precisely we have:

\begin{lemma}\label{lem:z-mu}
Let $V_{m} \in L^\infty(Y,{\Bbb R})$ be given for $m=1$ or $m=2$, and denote by $A_{m,\theta}$ the corresponding operator defined by \eqref{eq:Def-A-theta}.
For $f \in H^{1/2}_{\theta}(\partial Y)$ and $\mu \in {\Bbb C}$ and $\mu \notin {\rm sp}(A_{1,\theta}) \cup {\rm sp}(A_{2,\theta})$, let $u_{m,\mu} := u_{m}$ be the solution of \eqref{eq:u-m}. Then if $z_{\mu} := u_{1,\mu} - u_{2,\mu}$ we have
$$\|z_{\mu}\| + \|\nabla z_{\mu}\| + \|\Delta z_{\mu}\| \to 0\qquad \mbox{as }\,\mu \to -\infty. $$ 
In particular $\partial z_{\mu}/\partial{\bf n} \to 0$ in $L^2(\Gamma)$ as $\mu \to -\infty$.
\end{lemma}

\begin{proof}
It is enough to show that $\|z_{\mu}\| + \|\nabla z_{\mu}\| + \|\Delta z_{\mu}\| \to 0$ when $\mu \to -\infty$. Indeed, since $z_\mu\in\{v\in H^1_{0,\theta}(Y); \Delta v\in L^2(Y)\}=H^1_{0,\theta}(Y)\cap H^2(Y)$, the equality between these spaces resulting from classical regularity results for solutions to elliptic equations, we infer in particular that $\nabla z_\mu\cdot {\bf n} \in L^2(\Gamma)$ and
\begin{equation}\label{eq:Lem-JLL}
\|\nabla z_\mu\cdot {\bf n}\|_{L^2(\Gamma)} \leq c\, \left(\|\nabla z_{\mu}\| + \|\Delta z_{\mu}\| \right).
\end{equation}
One verifies that $z_{\mu}$ solves the equation
\begin{equation}\label{eq:z-mu}
\left\{ 
\begin{array}{rcll} 
-\Delta z_{\mu} + V_{1}z_{\mu} - \mu z_{\mu} & = & (V_{2} - V_{1})u_{2,\mu}, & \mbox{in}\ Y ,\\ 
z_{\mu}(\sigma) & = & 0,& \sigma\in \Gamma,\\
z_{\mu}(x',2\pi) & = & {\rm e}^{{\rm i}\theta} z_{\mu}(x',0) , & x'\in\omega,\\
\partial_{3} z_{\mu}(x',2\pi) & = & {\rm e}^{{\rm i}\theta}\partial_{3} z_{\mu}(x',0) , & x'\in\omega.
\end{array}\right.
\end{equation}
That is, denoting by $R_{1,\mu} = (A_{1,\theta} - \mu I)^{-1}$ the resolvent of the operator $A_{1,\theta}$, we have $z_{\mu} = R_{1,\mu}((V_{2}-V_{1})u_{2,\mu})$ and since, by Lemma \ref{lem:Resolution} we have $\|u_{2,\mu}\| \to 0$ as $\mu \to -\infty$, this yields that $\|z_{\mu}\| \to 0$. On the other hand, as $\mu \to -\infty$, for a constant independent of $\mu$ we have
$$\|\mu R_{1,\mu}((V_{2} - V_{1})u_{2,\mu})\| \leq c\, \|u_{2,\mu}\| \to 0,$$
and this implies that $\|\mu z_{\mu}\| \to 0$ in $L^2(Y)$ as $\mu \to -\infty$. From this, using the equation satisfied by $z_{\mu}$ we conclude that $-\Delta z_{\mu} + z_{\mu} \to 0$ in $L^2(Y)$ while $z_{\mu} \in H_{0,\theta}^1(Y)$. This implies  that $\|\nabla z_{\mu}\| \to 0$ and the proof of the lemma is complete.
\end{proof}

In the next section we give a representation formula for the one parameter family of Poincar\'e--Steklov operators $\Lambda_{\theta,V-\lambda}$, when the parameter $\lambda \in {\Bbb C}$ is appropriately chosen.

\section{A representation formula}
\label{sec:Isozaki}

In his paper going back to 1991, H.~Isozaki~\cite{Isozaki}, gives a simple representation formula which, in some sense, allows to express the potential $V$ in terms of the Poincar\'e--Steklov operator. More precisely, adapting the argument to fit our aim in this paper, let $\lambda \notin {\rm sp}(A_{\theta})$ and denote by $\Lambda_{\theta, V-\lambda}$ the Poincar\'e--Steklov operator defined by
$$f \mapsto {\partial u \over\partial{\bf n}}\quad\mbox{on }\, \Gamma,$$
where $u$ is the solution of equation \eqref{eq1}. For $\zeta = {\rm i}\xi + \eta \in {\Bbb C}^3$, where $\xi,\eta \in {\Bbb R}^3$, we shall denote
$$\zeta\cdot \zeta := - |\xi|^2 + |\eta|^2 + 2 {\rm i}\xi\cdot\eta,$$
where $\xi\cdot\eta$ denotes the usual scalar product of $\xi$ and $\eta$ in ${\Bbb R}^3$.
In the sequel we shall consider the functions $e_{\zeta}$ and $e_{*\zeta}$ defined below in terms of those $\zeta \in {\Bbb C}^3$ satisfying 
\begin{equation}\label{eq:Def-ezeta}
\zeta \in {\Bbb C}^3, \quad \zeta\cdot \zeta = -\lambda,
\quad e_{\zeta}(x) := \exp(\zeta\cdot x),\quad e_{*\zeta}(x) := \exp({{\overline \zeta}\cdot x}).
\end{equation}

\begin{definition}\label{def:S-V} Assume that $\zeta_{\ell} \in {\Bbb C}^3$ for $\ell = 0,1$ satisfy 
\eqref{eq:Def-ezeta} and are such that $e_{\zeta_{0}} \in H^1_{\theta}(Y)$ and $e_{*\zeta_{1}} \in H^1_{\theta}(Y)$. Then, following H. Isozaki, we set
$$S_{\theta,V}(\lambda,\zeta_{0},\zeta_{1}) := \int_{\Gamma} \Lambda_{\theta,V-\lambda}(e_{\zeta_{0}})(\sigma)\,e_{\zeta_{1}}(\sigma)\,d\sigma = \langle e_{*\zeta_{1}},\Lambda_{\theta,V-\lambda}(e_{\zeta_{0}}) \rangle. $$
\end{definition}

\begin{remark}
We point out that we should have defined the above function $S_{\theta,V}$ as being rather $\langle \Lambda_{\theta,V-\lambda}(e_{\zeta_{0}}) , e_{*\zeta_{1}} \rangle$, but  since one can easily check that $\Lambda_{\theta,V-\lambda}(e_{\zeta_{0}}) \in H^{1/2}(\partial Y)$, the way $S_{\theta,V}$ is defined above makes sense and is actually more convenient for our aims. \qed
\end{remark}

At this point let us make the following observations, which are going to be useful later. For $\zeta \in {\Bbb C}^3$ given, saying that $e_{\zeta}$ belongs to $H^1_{\theta}(Y)$ means that
\begin{equation}\label{eq:Cnd-zeta}
\zeta = {\rm i}\xi + \eta,\qquad \eta_{3} = 0, \quad\mbox{and for some }\,k\in {\Bbb Z},\quad \xi_{3} = {\theta \over 2\pi} + k.
\end{equation}
Clearly, an analogous observation holds for $e_{*\zeta}$ to belong to $H^1_{\theta}(Y)$: in this case $\zeta \in {\Bbb C}^3$ should verify:
\begin{equation}\label{eq:Cnd-zeta-star}
\zeta = {\rm i}\xi + \eta,\qquad \eta_{3} = 0, \quad\mbox{and for some }\,k\in {\Bbb Z},\quad \xi_{3} = {-\theta \over 2\pi} + k.
\end{equation}

Thanks to the function $(\lambda,\zeta_{0},\zeta_{1}) \mapsto S_{\theta,V}(\lambda,\zeta_{0},\zeta_{1})$ defined above, we have the following result: the Fourier transform of the potential $V$ can be expressed in terms of $S_{\theta, V}$, that is the Poincar\'e--Steklov operator applied to $e_{\zeta_{\ell}}$:

\begin{lemma}\label{lem:Isozaki} 
Let $\lambda \notin {\rm sp}(A_{\theta})$ and denote by $R_{\lambda} := (A_{\theta} - \lambda)^{-1}$ the resolvent of $A_{\theta}$ acting on $L^2(Y)$. 
Let $\zeta_{0}\in {\Bbb C}^3$ satisfy \eqref{eq:Cnd-zeta} and $\zeta_{1} \in {\Bbb C}^3$ satisfy \eqref{eq:Cnd-zeta-star}, so that $e_{\zeta_{0}},e_{*\zeta_{1}} \in H^1_{\theta}(Y)$.

Then we have the following representation formula:
\begin{eqnarray}
&\displaystyle \int_{Y} \!\! V(x)\, {\rm e}^{(\zeta_{0} + \zeta_{1})\cdot x}\, dx = S_{\theta,V}(\lambda,\zeta_{0},\zeta_{1}) - \displaystyle (\zeta_{0}\cdot \zeta_{1} - \lambda) \int_{Y}\!\! {\rm e}^{(\zeta_{0} + \zeta_{1})\cdot x}dx 
\nonumber \\
& \;\qquad\qquad\qquad \displaystyle + \int_{Y}R_{\lambda}(Ve_{\zeta_{0}})(x)V(x) e_{\zeta_{1}}(x)\,dx.
\label{eq:Id-Poinca-Steklov}
\end{eqnarray}
\end{lemma}

\begin{proof}
This result, due to H.~Isozaki \cite{Isozaki} in another situation, can be proved exactely in the same way as in \cite{Isozaki}, but the boundary conditions involved here are slightly different and thus the test functions $e_{\zeta_{\ell}}$ have to be of a certain type, hence for the reader's convenience, and the sake of completeness, we give its proof.
Let $u \in H^1_{\theta}(Y)$ be the unique solution of the equation
\begin{equation}\label{eq:u-e-zeta}
\left\{ 
\begin{array}{rcll} 
-\Delta u + Vu - \lambda u & = & 0, & \mbox{in}\ Y ,\\ 
u(\sigma) & = & e_{\zeta_{0}}(\sigma),& \sigma\in \Gamma,\\
 u(x',2\pi) & = & {\rm e}^{{\rm i}\theta} u(x',0) , & x'\in\omega,\\
\partial_{3} u(x',2\pi) & = & {\rm e}^{{\rm i}\theta}\partial_{3} u(x',0) , & x'\in\omega.
\end{array}\right.
\end{equation}
Taking the scalar product of the first equation with the function $e_{*\zeta_{1}}\in H^1_{\theta}(Y)$, we have that
\begin{equation}\label{eq:Id-1}
-\int_{Y}\Delta u(x) \overline{e_{*\zeta_{1}}}(x)\,dx + \int_{Y}(V(x) - \lambda)u(x)\,\overline{e_{*\zeta_{1}}}(x) \,dx =0, 
\end{equation}
and since with our notations $\overline{e_{*\zeta_{1}}}(x) = e_{\zeta_{1}}(x)$, and 
$-\Delta e_{\zeta_{1}}(x) = \lambda e_{\zeta_{1}}(x)$,
after two integration by parts  we get 
\begin{eqnarray}
-\int_{Y}\!\!\Delta u(x) e_{\zeta_{1}}(x)\,dx &= \displaystyle - \int_{\partial Y}\!\!\left[{\partial u(\sigma) \over \partial{\bf n}} \, e_{\zeta_{1}}(\sigma) 
- u(\sigma)\, {\partial e_{\zeta_{1}}(\sigma) \over\partial{\bf n} }\right]\, d\sigma \nonumber\\
&\displaystyle +\lambda\int_{Y}\!\!u(x)\,e_{\zeta_{1}}(x)\, dx. \label{eq:Int-by-parts-1}
\end{eqnarray}
Now, taking into account the fact that $u$ and $e_{*\zeta_{1}}$ belong to $H^1_{\theta}(Y)$, while (note the presence of ${\rm e}^{-{\rm i}\theta}$ in front of $\partial_{3}e_{\zeta_{1}}(x',0)$)
$$\partial_{3}u(x',2\pi) - {\rm e}^{{\rm i}\theta}\partial_{3}u(x',0) = \partial_{3}e_{\zeta_{1}}(x',2\pi) - {\rm e}^{-{\rm i}\theta}\partial_{3} e_{\zeta_{1}}(x',0) = 0 \quad\mbox{for }\, x'\in \omega,$$
we conclude that
$$\int_{\partial Y \setminus \Gamma}\left[{\partial u(\sigma) \over \partial{\bf n}} \, e_{\zeta_{1}}(\sigma) 
- u(\sigma)\, {\partial e_{\zeta_{1}}(\sigma) \over\partial{\bf n} }\right]\, d\sigma =0, $$
and finally from \eqref{eq:Id-1} and \eqref{eq:Int-by-parts-1} we induce
$$\int_{Y}\!\!V(x) \,u(x)\,e_{\zeta_{1}}(x)\, dx = 
\int_{\Gamma}\left[{\partial u(\sigma) \over \partial{\bf n}} \, e_{\zeta_{1}}(\sigma) 
- u(\sigma)\, {\partial e_{\zeta_{1}}(\sigma) \over\partial{\bf n} }\right]\, d\sigma .$$
Now, since $\partial u/\partial{\bf n} = \Lambda_{\theta,V-\lambda}(e_{\zeta_{0}})$ and $u(\sigma)= e_{\zeta_{0}}(\sigma)$ on $\Gamma$, the above can be written as
\begin{equation}
\int_{Y}\!\!V \,u\,e_{\zeta_{1}}\, dx  = 
\int_{\Gamma}\Lambda_{\theta,V-\lambda}(e_{\zeta_{0}})(\sigma) \, e_{\zeta_{1}}(\sigma)\, d\sigma 
- \int_{\Gamma} e_{\zeta_{0}}(\sigma)\, {\partial e_{\zeta_{1}}(\sigma) \over\partial{\bf n} }\, d\sigma. \label{eq:Int-by-parts-2}
\end{equation}
To inspect further the left hand side of \eqref{eq:Int-by-parts-2}, we write $u = \psi + e_{\zeta_{0}}$ for some $\psi \in H^1_{0,\theta}(Y)$: in fact $\psi \in D(A_{\theta})$ and satisfies the equation
$$-\Delta \psi + V\psi -\lambda \psi = - Ve_{\zeta_{0}},$$
which means that 
\begin{equation}\label{eq:Diff-u-e-zeta}
u - e_{\zeta_{0}} = \psi = -R_{\lambda}(Ve_{\zeta_{0}}).
\end{equation}
 Therefore we have 
\begin{equation*}
\int_{Y}\!V \,u\,e_{\zeta_{1}}\, dx = \int_{Y}\! V\,e_{\zeta_{0}}\, e_{\zeta_{1}}\, dx - \int_{Y}\! Ve_{\zeta_{1}}R_{\lambda}(Ve_{\zeta_{0}})\,dx,
\end{equation*}
that is 
\begin{equation}\label{eq:Rel-V-u}
\int_{Y}\!V \,u\,e_{\zeta_{1}}\, dx = \int_{Y}\! V(x) \, {\rm e}^{(\zeta_{0}+\zeta_{1})\cdot x}dx - \int_{Y}\! V(x)e_{\zeta_{1}}(x)R_{\lambda}(Ve_{\zeta_{0}})(x)\,dx.
\end{equation}

Regarding the second term of the right hand side of the equality \eqref{eq:Int-by-parts-2}, upon observing that $e_{\zeta_{1}}$ satisfies the equation
$-\Delta e_{\zeta_{1}} = \lambda e_{\zeta_{1}}$, multiplying this equality by $e_{\zeta_{0}}$ and integrating by parts, we get
$$\int_{\partial Y}e_{\zeta_{0}}(\sigma)\, {\partial e_{\zeta_{1}}(\sigma) \over \partial{\bf n}}\, d\sigma = (\zeta_{0}\cdot \zeta_{1} - \lambda) \int_{Y}{\rm e}^{(\zeta_{0} + \zeta_{1})\cdot x}\, dx.$$
But since $\zeta_{0}$ satisfies \eqref{eq:Cnd-zeta} and $\zeta_{1}$ satisfies \eqref{eq:Cnd-zeta-star}, we have that
$$\int_{\partial Y \setminus \Gamma}e_{\zeta_{0}}(\sigma)\, {\partial e_{\zeta_{1}}(\sigma) \over \partial{\bf n}}\, d\sigma = 0,$$
and finally we obtain
$$\int_{\Gamma}e_{\zeta_{0}}(\sigma)\, {\partial e_{\zeta_{1}}(\sigma) \over \partial{\bf n}}\, d\sigma = (\zeta_{0}\cdot \zeta_{1} - \lambda) \int_{Y}{\rm e}^{(\zeta_{0} + \zeta_{1})\cdot x}\, dx.$$
Plugging this, together with \eqref{eq:Rel-V-u}, into \eqref{eq:Int-by-parts-2} we obtain the identity claimed in the theorem.
\end{proof}

From \eqref{eq:Id-Poinca-Steklov} we see that if for any given $\xi\in {\Bbb R}^3$ we can find $\zeta_{0},\zeta_{1}$ such that on the one hand $\zeta_{0} + \zeta_{1} \sim -{\rm i}\xi$, and on the other hand they are such that the term
$$\int_{Y}R_{\lambda}(Ve_{\zeta_{0}})(x)\,V(x)e_{\zeta_{1}}(x)\,dx$$
is small, then the Fourier transform of $1_{Y}V$ is known in terms of $S_{\theta,V}(\lambda,\zeta_{0},\zeta_{1})$, up to the above error term. To this end we establish the following couple of lemmas, where we construct appropriate directions $\zeta_{0},\zeta_{1}$. The first lemma concerns the case of integer Fourier frequency not less than 1 (in the direction $x_{3}$) :

\begin{lemma}\label{lem:Larger-than-zero}
Let $\xi' \in {\Bbb R}^2$, with $|\xi'| \neq 0$, and $j \in {\Bbb Z}$ with $j \geq 1$. We fix $\eta' \in {\Bbb R}^2$ such that $|\eta'| = 1$ and $\xi'\cdot \eta' = 0$. For $s \in (0,1)$ and $(1-s)$ sufficiently small, we define for $t >0$ such that $st > |\xi'|/2$:
\begin{eqnarray*}
& \zeta'_{0} := \displaystyle {\rm i}(st + {\rm i})\left({-\xi' \over 2st} + \left(1 - {|\xi'|^2 \over 4s^2t^2}\right)^{1/2} \eta' \right) \\
& \zeta'_{1} := \displaystyle {\rm i}(t + {\rm i}s)\left({-\xi' \over 2t} - \left(1 - {|\xi'|^2 \over 4t^2}\right)^{1/2} \eta' \right) ,
\end{eqnarray*}
and set
\begin{equation}
\zeta_{0} := \left(\zeta'_{0},{\rm i}\left(j + {\theta \over 2\pi}\right)\right)\in {\Bbb C}^3,\qquad
\zeta_{1} := \left(\zeta'_{1}, {-{\rm i}\theta \over 2\pi}\right) \in {\Bbb C}^3.
\end{equation}
Upon choosing $t > 0$ such that
$$t^2 = {1 \over 1 - s^2}\, j\left(j + {\theta \over \pi}\right) - 1,$$
we have $\lambda := - \zeta_{0}\cdot\zeta_{0} = - \zeta_{1}\cdot\zeta_{1}$. Moreover as $s \to 1$ we have the following
\begin{eqnarray}
& t \to + \infty, \qquad {\rm Im}(\lambda) \to + \infty, \qquad \zeta_{0} + \zeta_{1} \to -{\rm i}(\xi',-j),\\
& \displaystyle  \zeta_{0} \cdot\zeta_{1} - \lambda \to -{1 \over 2} \left(|\xi'|^2 + j^2\right)
\end{eqnarray}
We have also $\|e_{\zeta_{0}}\| + \|e_{\zeta_{1}}\| \leq c$  for some constant $c > 0$ independent of $s$.
\end{lemma}

\begin{proof}
The fact that with the above choice of $t$ one has $\zeta_{0}\cdot\zeta_{0} = \zeta_{1}\cdot \zeta_{1}$ is just a matter of elementary algebra since
$$-\zeta_{0}\cdot\zeta_{0} = (st + {\rm i})^2 + \left(j + {\theta \over 2\pi}\right)^2,\qquad -\zeta_{1}\cdot\zeta_{1} = (t + {\rm i}s)^2 + {\theta^2 \over 4\pi^2},$$
and one verifies that $\zeta_{0}\cdot\zeta_{0} = \zeta_{1}\cdot\zeta_{1}$. Calling this commun value $-\lambda := \zeta_{1}\cdot\zeta_{1}$, we have that ${\rm Re}(\lambda) = t^2 - s^2 + \theta^2/4\pi^2$, while ${\rm Im}(\lambda) = 2st$. Note that
$$\zeta_{0}\cdot\zeta_{1} - \lambda = {1 \over 2}(\zeta_{0} + \zeta_{1}) \cdot (\zeta_{0} + \zeta_{1}) \to -{1 \over 2} \left(|\xi'|^2 + j^2\right),$$
since it is clear that when $s \to 1$ we have $t \to +\infty$, and also as a consequence ${\rm Im}(\lambda) = 2st \to + \infty$. 
\end{proof}

The second lemma concerns the case of the integer Fourier frequency (in the direction $x_{3}$) not larger than $-1$:

\begin{lemma}\label{lem:Less-than-zero}
Let $\xi' \in {\Bbb R}^2$, with $|\xi'| \neq 0$, and $j \in {\Bbb Z}$ with $j \leq - 1$. We fix $\eta' \in {\Bbb R}^2$ such that $|\eta'| = 1$ and $\xi'\cdot \eta' = 0$. For $s \in (0,1)$ and $(1-s)$ sufficiently small, we define for $t >0$ such that $st > |\xi'|/2$:
\begin{eqnarray*}
& \zeta'_{0} := \displaystyle {\rm i}(st + {\rm i})\left({-\xi' \over 2st} + \left(1 - {|\xi'|^2 \over 4s^2t^2}\right)^{1/2} \eta' \right) \\
& \zeta'_{1} := \displaystyle {\rm i}(t + {\rm i}s)\left({-\xi' \over 2t} - \left(1 - {|\xi'|^2 \over 4t^2}\right)^{1/2} \eta' \right) ,
\end{eqnarray*}
and set
\begin{equation}
\zeta_{0} := \left(\zeta'_{0},{\rm i}\left(j - 1 + {\theta \over 2\pi}\right)\right)\in {\Bbb C}^3,\qquad
\zeta_{1} := \left(\zeta'_{1}, {\rm i}\left( 1 - {\theta \over 2\pi}\right)\right) \in {\Bbb C}^3.
\end{equation}
Upon choosing $t > 0$ such that
$$t^2 = {1 \over 1 - s^2}\, j\left(j - 2 + {\theta \over \pi}\right) - 1,$$
we have $\lambda := - \zeta_{0}\cdot\zeta_{0} = - \zeta_{1}\cdot\zeta_{1}$. Moreover as $s \to 1$ we have the following
\begin{eqnarray}
& t \to + \infty,\qquad {\rm Im}(\lambda) \to + \infty, \qquad \zeta_{0} + \zeta_{1} \to -{\rm i}(\xi',-j),\\
& \displaystyle \zeta_{0} \cdot\zeta_{1} - \lambda \to -{1 \over 2} \left(|\xi'|^2 + j^2\right).
\end{eqnarray}
We have also $\|e_{\zeta_{0}}\| + \|e_{\zeta_{1}}\| \leq c$  for some constant $c > 0$ independent of $s$.
\end{lemma}

Finally we state how to choose the vectors $\zeta_{0}, \zeta_{1} \in {\Bbb C}^3$ when the Fourier frequency $j$ in the direction $x_{3}$ is equal to zero:

\begin{lemma}\label{lem:Equal-to-zero}
Let $\xi' \in {\Bbb R}^2$, with $|\xi'| \neq 0$, and $j = 0$. We fix $\eta' \in {\Bbb R}^2$ such that $|\eta'| = 1$ and $\xi'\cdot \eta' = 0$. For $t >0$ such that $t > |\xi'|/2$ we set:
\begin{eqnarray*}
& \zeta'_{0} := \displaystyle {\rm i}(t + {\rm i})\left({-\xi' \over 2t} + \left(1 - {|\xi'|^2 \over 4t^2}\right)^{1/2} \eta' \right) \\
& \zeta'_{1} := \displaystyle {\rm i}(t + {\rm i})\left({-\xi' \over 2t} - \left(1 - {|\xi'|^2 \over 4t^2}\right)^{1/2} \eta' \right) ,
\end{eqnarray*}
and set
\begin{equation}
\zeta_{0} := \left(\zeta'_{0},{\rm i}\,{\theta \over 2\pi}\right)\in {\Bbb C}^3,\qquad
\zeta_{1} := \left(\zeta'_{1}, -{\rm i}\, {\theta \over 2\pi}\right) \in {\Bbb C}^3.
\end{equation}
Then we have $\lambda := - \zeta_{0}\cdot\zeta_{0} = - \zeta_{1}\cdot\zeta_{1} = (t + {\rm i})^2 - (\theta^2/4\pi^2)$. Moreover as $t \to \infty$ we have the following
\begin{eqnarray}
&{\rm Im}(\lambda) \to + \infty, \qquad \zeta_{0} + \zeta_{1} \to -{\rm i}(\xi',0),\\
& \displaystyle \zeta_{0} \cdot\zeta_{1} - \lambda \to -{1 \over 2} \,|\xi'|^2 .
\end{eqnarray}
We have also $\|e_{\zeta_{0}}\| + \|e_{\zeta_{1}}\| \leq c$  for some constant $c > 0$ independent of $t$.
\end{lemma}

The proof of this lemma is straightforward and can be omitted.

\medskip

Since with the above choices of $\zeta_{0},\zeta_{1}$, saying that $s \to 1$, or $t\to \infty$, is equivalent to $|\lambda| \to +\infty$, we shall write only the latter. Now it is easy to obtain the following representation formula:

\begin{theorem}\label{thm:Isozaki}
Let $j \in {\Bbb Z}$ and for any $\xi'\in {\Bbb R}^2$, with $|\xi'|\neq 0$, set $\xi := (\xi',j) \in {\Bbb R}^3$. According to whether $j \geq 1$ or $j \leq -1$ or $j = 0$, let $\lambda,\zeta_{0},\zeta_{1}$ be given by either Lemma \ref{lem:Larger-than-zero}, or Lemma \ref{lem:Less-than-zero}, or Lemma \ref{lem:Equal-to-zero}. Then, as $|\lambda| \to +\infty$ we have
\begin{equation}\label{eq:Fourier-of-V}
\int_{Y}\! V(x)\, {\rm e}^{-{\rm i}\,\xi\cdot x}\, dx = 
\lim_{|\lambda| \to +\infty}\! S_{\theta,V}(\lambda,\zeta_{0},\zeta_{1}) + 
{|\xi|^2 \over 2} \int_{Y}\! {\rm e}^{-{\rm i}\,\xi\cdot x}\, dx.
\end{equation}
\end{theorem}

\begin{proof} The left hand side of the identity \eqref{eq:Id-Poinca-Steklov} converges clearly to the left hand side of what is claimed in relation \eqref{eq:Fourier-of-V}.

Now, on the one hand the second term in the right hand side of \eqref{eq:Id-Poinca-Steklov} converges to zero, since $(\zeta_{0}\cdot\zeta_{1} - \lambda)$ has a finite limit, while 
$$\int_{Y}{\rm e}^{(\zeta_{0} + \zeta_{1})\cdot x}\, dx \to \int_{Y}\exp(-{\rm i}\,\xi'\cdot x' + {\rm i}\,jx_{3})\, dx.$$
(Note that when $j \in {\Bbb Z}^*$ the latter integral is equal to zero, because $Y = \omega\times(0,2\pi)$ and we have $\int_{0}^{2\pi} \exp({\rm i}\,jx_{3})\, dx_{3} = 0$).

On the other hand, we have $\|e_{\zeta_{0}}\| + \|e_{\zeta_{1}}\| \leq c$ as $|\lambda| \to +\infty$, and we may remind that the resolvent $R_{\lambda}$ satisfies
$$\|R_{\lambda}\| \leq {c \over {\rm dist}(\lambda, {\rm sp}(A_{\theta}))} \to 0,$$
because ${\rm dist}(\lambda, {\rm sp}(A_{\theta}) ) \geq {\rm Im}(\lambda) \to + \infty$. Therefore the third term in the right hand side of \eqref{eq:Id-Poinca-Steklov} converges to zero, and the proof of the theorem is complete. 
\end{proof}

For later use, we state the following result regarding the behavior of $\zeta'_{0} + \zeta'_{1}$ and $e_{\zeta_{0}},e_{\zeta_{1}}$ as $|\lambda| \to \infty$: its proof needs only a close examination of the definitions of $\zeta'_{0},\zeta'_{1}$. 

\begin{lemma}\label{lem:Behaviour-zeta}
Let $j \in {\Bbb Z}$ and for any $\xi'\in {\Bbb R}^2$, with $|\xi'|\neq 0$, set $\xi := (\xi',j) \in {\Bbb R}^3$. According to whether $j \geq 1$ or $j \leq -1$ or $j = 0$, let $\lambda,\zeta_{0},\zeta_{1}$ be given by either Lemma \ref{lem:Larger-than-zero}, or Lemma \ref{lem:Less-than-zero}, or Lemma \ref{lem:Equal-to-zero}. Then, as $|\lambda| \to +\infty$ we have
\begin{equation}\label{eq:Behaviour-zeta}
\zeta_{0} + \zeta_{1} - {\rm i}(\xi',-j) = O\left(|\lambda|^{-1/2}\right),\qquad
\|e_{\zeta_{0}}\|_{L^2(\Gamma)} + \|e_{\zeta_{1}}\|_{L^2(\Gamma)} = O(1).
\end{equation}
\end{lemma}

\begin{remark}
While $\lambda,\zeta_{0},\zeta_{1}$ are as in the above lemma, let $u$ be the solution of \eqref{eq:u-e-zeta} with $\zeta = \zeta_{0}$. Since on the one hand $\|e_{\zeta}\|_{L^2(Y)}$ is bounded, and on the other hand, thanks to \eqref{eq:Diff-u-e-zeta}, we have $u - e_{\zeta} = -R_{\lambda}(Ve_{\zeta})$, so that  $\|u-e_{\zeta}\|_{L^2(Y)} \to 0$ as $|\lambda|\to\infty$, because $\|R_{\lambda}\|_{L^2(Y)\to L^2(Y)} \to 0$. Therefore, for a positive constant $c$, we have $\|u\|_{L^2} \leq c$.

In the same manner, if instead of \eqref{eq:u-e-zeta} we consider the following equation
\begin{equation}\label{eq:u*-e-zeta}
\left\{ 
\begin{array}{rcll} 
-\Delta u_{*} + Vu_{*} - \overline{\lambda} u_{*} & = & 0, & \mbox{in}\ Y ,\\ 
u_{*}(\sigma) & = & e_{\zeta_{*1}}(\sigma),& \sigma\in \Gamma,\\
 u_{*}(x',2\pi) & = & {\rm e}^{{\rm i}\theta} u_{*}(x',0) , & x'\in\omega,\\
\partial_{3} u_{*}(x',2\pi) & = & {\rm e}^{{\rm i}\theta}\partial_{3} u_{*}(x',0) , & x'\in\omega,
\end{array}\right.
\end{equation}
since $e_{*\zeta_{1}} \in H^1_{\theta}(Y)$ and $-\Delta e_{*\zeta_{1}} = \overline{\lambda}e_{\zeta_{1}}$, we may apply the same arguments invoked for $u$, and conclude that $\|u_{*}\|_{L^2(Y)} \leq c$ for some positive constant.
\end{remark}

These above observations will be used in Section \ref{sec:Proof-main}, therefore for the reader's convenience we state the following lemma:

\begin{lemma}\label{lem:Estim-u-e-zeta}
Let $M,\lambda,\zeta_{0},\zeta_{1}$, as well as $V$, be as in Lemma \ref{lem:Behaviour-zeta}, and let $u$ be the solution of \eqref{eq:u-e-zeta}, and $u_{*}$ be that of \eqref{eq:u*-e-zeta}. Then, for a constant depending only on $M$ and $\omega$, 
\begin{equation}\label{eq:Estim-u-e-zeta}
\sum_{k \geq 1} {|\langle \psi_{k},e_{\zeta_{0}}\rangle|^2 \over |\lambda - \lambda_{k}|^2} 
+ \sum_{k \geq 1} {|\langle \psi_{k},e_{*\zeta_{1}}\rangle|^2 \over |\overline{\lambda} - \lambda_{k}|^2}
= \|u\|_{L^2(Y)}^2 + \|u_{*}\|_{L^2(Y)}^2 \leq c.
\end{equation}
\end{lemma}

\medskip

In the next section we use the representation formula \eqref{eq:Fourier-of-V}, together with the observations made in Section \ref{sec:Isozaki}, to prove the main results stated in Theorems \ref{thm:Uniqueness-QP}--\ref{thm:Uniqueness-Asymptotic}.


\section{Proof of the main results}
\label{sec:Proof-main}

In this entire section, we write $\theta$ instead of $\theta_0$. 
For $m = 1$ and $m = 2$ consider two potentials $V_{m}$ satisfying the assumptions of Theorem \ref{thm:Uniqueness-QP}, and denote by $A_{m,\theta}$ the associated operators defined by \eqref{eq:Def-A-theta}. 

Let $(\lambda_{m,k}(\theta),\phi_{m,\theta,k})_{k \geq 1}$ be the sequence of eigenvalues and eigenfunctions of $A_{m,\theta}$.
Once the value of $\theta$, or $\theta_{0}$, is fixed, for the sake of the simplicity of notations, we shall denote the operators, and their eigenvalues and eigenfunctions as 
$$A_{m} := A_{m,\theta},\quad
\lambda_{m,k} := \lambda_{m,k}(\theta),\quad
\mbox{and}\quad
\phi_{m,k} := \phi_{m,\theta,k},\quad
\psi_{m,k} := \psi_{m,\theta,k}.$$ 
Consider now $\lambda \in {\Bbb C}$ and $\mu \in {\Bbb R}$, such that $\lambda,\mu \notin {\rm sp}(A_{1,\theta}) \cup {\rm sp}(A_{2,\theta})$. For  $f \in H_{\theta}^{1/2}(\Gamma)$ denote by $u_{m,\lambda}$ the solution to  \eqref{eq1} where $V := V_{m}$, and also recall that
$$u_{m,\lambda} = \sum_{k\geq 1}{\alpha_{m,k} \over \lambda - \lambda_{m,k}}\qquad\mbox{where}\qquad
\alpha_{m, k} := \langle \psi_{m,k},f\rangle.$$

We recall also that thanks to the variational characterization of the eigenvalues of the self-adjoint operators $A_{m}$, one can easily see that 
\begin{equation}\label{eq:Estim-ecart-vp}
|\lambda_{1,k} - \lambda_{2,k}| \leq  \|V_{1} - V_{2}\|_{\infty}.
\end{equation}
Next we introduce a few notations and conventions in order to make the proofs more clear. 

We split ${\Bbb N}^*$ into two subsets of integers $k \geq 1$, according to whether we have $(\lambda_{1,k},\psi_{1,k}) = (\lambda_{2,k},\psi_{2,k})$ or not: more precisely we set
\begin{equation}\label{eq:Def-K-0}
{\Bbb K}_{1} := \left\{ k \geq 1 \; ; \; (\lambda_{1,k},\psi_{1,k}) = 
(\lambda_{2,k},\psi_{2,k}) \right\}, \qquad\mbox{and}\qquad
{\Bbb K}_{0} := {\Bbb N}^* \setminus {\Bbb K}_{1}.
\end{equation}
Moreover, when $k \in {\Bbb K}_{1}$ we drop the index $m=1$ or $m=2$ for the eigenvalues and the normal derivatives of the eigenfunctions, that is we denote by $\lambda_{k}$ and $\psi_{k}$, as well as $\alpha_{k}$, the common value of these entities.
Now, $f$ being fixed, with the notations of Lemmas \ref{lem:Resolution} and \ref{lem:v-lambda-mu} we denote by $v_{m,\lambda,\mu} := u_{m,\lambda} - u_{m,\mu}$ the solution of \eqref{eq:v-lambda} where $V$ is replaced by $V_{m}$, and we set
\begin{equation}\label{eq:Def-F-m}
F_{m}(\lambda,\mu,f) := \sum_{k \in {\Bbb K}_{0}} {(\mu - \lambda)\alpha_{m,k}  \over (\lambda - \lambda_{m,k})(\mu - \lambda_{m,k})} \, \psi_{m,k}
\end{equation}
and analogously (note that for $k\in {\Bbb K}_{1}$ we write $\lambda_{k} := \lambda_{1,k}=\lambda_{2,k}$ and $\psi_{k} := \psi_{1,k} = \psi_{2,k}$)
\begin{equation}\label{eq:Def-G}
G(\lambda,\mu,f) := \sum_{k \in {\Bbb K}_{1}} {(\mu - \lambda)\alpha_{k}\over (\lambda - \lambda_{k})(\mu - \lambda_{k})} \, \psi_{k} .
\end{equation}
Using the notations introduced above, and according to \eqref{eq:v-Normal-Deriv} in Lemma \ref{lem:v-lambda-mu},  we have
\begin{equation}\label{eq:Deriv-Normal-m}
{\partial v_{m,\lambda,\mu} \over \partial{\bf n}} = F_{m}(\lambda,\mu,f) + G(\lambda,\mu,f).
\end{equation}

Recall that in Lemma \ref{lem:z-mu} we have set $z_{\mu} = u_{1,\mu} - u_{2,\mu}$, and thus writing the above identity \eqref{eq:Deriv-Normal-m} for $m=1$ and $m=2$, and then subtracting the resulting equations, we end up with a new relation, namely
\begin{equation}\label{eq:Id-lambda-mu}
{\partial u_{1,\lambda} \over \partial{\bf n}} - 
{\partial u_{2,\lambda} \over \partial{\bf n}} = 
{\partial z_{\mu} \over \partial{\bf n}} 
+ F_{1}(\lambda,\mu,f) - F_{2}(\lambda,\mu,f).
\end{equation}
It is convenient to set
\begin{equation}\label{eq:Def-F-*}
F_{*m}(\lambda,f) := \sum_{k \in {\Bbb K}_{0}} {\alpha_{m,k} \over \lambda - \lambda_{m,k}} \, \psi_{m,k}, 
\end{equation}
so that while $\lambda$ and $f$ are fixed, upon letting $\mu \to -\infty$, first we shall prove that
$$F_{m}(\lambda,\mu,f) \to F_{*m}(\lambda,f),$$
in an appropriate sense (see below Lemma \ref{lem:Lim-mu-F1-F2}). Then, thanks to Lemma \ref{lem:z-mu}, as $\mu\to -\infty$, we obtain from \eqref{eq:Id-lambda-mu} that
\begin{equation}\label{eq:Id-lambda}
{\partial u_{1,\lambda} \over \partial{\bf n}} - 
{\partial u_{2,\lambda} \over \partial{\bf n}} = 
F_{*1}(\lambda,f) - F_{*2}(\lambda,f).
\end{equation}
Now choose $\lambda, \zeta_{0},\zeta_{1}$ as in Theorem \ref{thm:Isozaki}, and $f := e_{\zeta_{0}}$. Then the identity \eqref{eq:Id-lambda} and Definition \ref{def:S-V} yield
\begin{equation}\label{eq:S-F-*}
S_{\theta,V_{1}}(\lambda,\zeta_{0},\zeta_{1}) - S_{\theta,V_{2}}(\lambda,\zeta_{0},\zeta_{1}) = 
\langle  e_{*\zeta_{1}}, F_{*1}(\lambda,e_{\zeta_{0}}) - 
F_{*2}(\lambda,e_{\zeta_{0}}) \rangle.
\end{equation}
On the other hand, thanks to \eqref{eq:Fourier-of-V} of Theorem \ref{thm:Isozaki}, recall  that we have
\begin{equation}\label{eq:V-hat-and-S1-S2}
\int_{Y}\! (V_{1}-V_{2})(x)\, {\rm e}^{-{\rm i}\,\xi\cdot x}\, dx = 
\lim_{|\lambda| \to +\infty}\left(S_{\theta,V_{1}}(\lambda,\zeta_{0},\zeta_{1}) - S_{\theta,V_{2}}(\lambda,\zeta_{0},\zeta_{1})\right).\end{equation}
One sees that in order to obtain the Fourier transform of $V_{1}-V_{2}$, we have to determine the limit in \eqref{eq:S-F-*} as $|\lambda|\to \infty$. 
Now, since the right hand side of \eqref{eq:S-F-*} is given by the limit of the series
\begin{eqnarray}
&\langle e_{*\zeta_{1}}, F_{1}(\lambda,\mu,e_{\zeta_{0}}) - F_{2}(\lambda,\mu,e_{\zeta_{0}}) \rangle &= \label{eq:Diff-F1-F2}\\
& &\hskip-3.5cm\displaystyle(\mu - \lambda)\sum_{k \in {\Bbb K}_{0}}\left[ {\langle \psi_{1,k}, e_{\zeta_{0}}\rangle\, \langle e_{*\zeta_{1}},\psi_{1,k} \rangle  \over (\lambda - \lambda_{1,k})(\mu - \lambda_{1,k})}
- {\langle \psi_{2,k},e_{\zeta_{0}}\rangle\, \langle e_{*\zeta_{1}},\psi_{2,k} \rangle  \over (\lambda - \lambda_{2,k})(\mu - \lambda_{2,k})}\right], \nonumber
\end{eqnarray}
as $\mu \to -\infty$, we have to investigate under which assumptions we may find the limit of the above series as $\mu\to-\infty$ and $|\lambda|\to\infty$: once this is done, then \eqref{eq:V-hat-and-S1-S2} yields the appropriate interpretaion about the Fourier transform of $V_{1}-V_{2}$.

\medskip

We introduce two functions $f_{\lambda,\mu} : [-M,+\infty) \dans {\Bbb C}$ and $g := L^2(\Gamma) \dans {\Bbb C}$ such that, for each $k \geq 1$ fixed and $m=1$ or $m=2$, the terms of the series can be written as
$$(\mu-\lambda){\langle \psi_{m,k}, e_{\zeta_{0}}\rangle\, \langle e_{*\zeta_{1}},\psi_{m,k} \rangle  \over (\lambda - \lambda_{m,k})(\mu - \lambda_{m,k})} =
f_{\lambda,\mu}(\lambda_{m,k})\,g(\psi_{m,k}).$$
(Here recal that $M > 0$ is such that $\|V_{m}\|_{\infty} \leq M$). So in a first step we state and show the following elementary lemmas.

\begin{lemma}\label{lem:Estim-1}
For ${\rm Im}(\lambda) \geq 1$ and $\mu \leq - (M+1)$ where $M \geq 0$, let $f_{\lambda,\mu}: [-M,+\infty) \to {\Bbb C}$ defined by
$$f_{\lambda,\mu}(\tau) := {\mu - \lambda \over (\lambda - \tau)(\mu - \tau)}.$$
Then for $\tau_{1} , \tau_{2} \geq -M$ we have
\begin{equation}\label{eq:Estim-1}
|f_{\lambda,\mu}(\tau_{1}) - f_{\lambda,\mu}(\tau_{2})| \leq 2\,|\tau_{1} - \tau_{2}|\, \max_{\tau \in [\tau_{1},\tau_{2}]}\left[{1 \over |\lambda -\tau|^2} + {1 \over |\mu - \tau|^2}\right].
\end{equation}
\end{lemma}

\begin{proof}
Indeed, assuming for instance $\tau_{1} < \tau_{2}$, we have 
$$f_{\lambda,\mu}(\tau_{2}) - f_{\lambda,\mu}(\tau_{1}) =  \int_{\tau_{1}}^{\tau_{2}}f_{\lambda,\mu}'(\tau)\,d\tau.$$
This yields
$$|f_{\lambda,\mu}(\tau_{1}) - f_{\lambda,\mu}(\tau_{2})| \leq |\tau_{1} - \tau_{2}|\, \max_{\tau \in [\tau_{1},\tau_{2}]} {|\lambda -\mu|\left(|\lambda -\tau| + |\mu -\tau|\right) \over |\lambda -\tau|^2\cdot |\mu - \tau|^2 }.$$
Now it is clear that
$${|\lambda -\mu|\left(|\lambda -\tau| + |\mu -\tau|\right) \over |\lambda -\tau|^2\cdot |\mu - \tau|^2 } = 
{|\lambda -\mu| \over |\lambda -\tau|\cdot |\mu - \tau|^2 } +
{|\lambda -\mu| \over |\lambda -\tau|^2\cdot |\mu - \tau| },$$
and using the inequality $|\lambda - \mu| \leq |\lambda - \tau| + |\mu - \tau|$ we end up with
$${|\lambda -\mu|\left(|\lambda -\tau| + |\mu -\tau|\right) \over |\lambda -\tau|^2\cdot |\mu - \tau|^2 } \leq 
{1 \over |\mu - \tau|^2 } + {2 \over |\lambda -\tau|\cdot |\mu - \tau| }
+ {1\over |\lambda -\tau|^2 },$$
which, using the fact that 
$$ {2 \over |\lambda -\tau|\cdot |\mu - \tau| } \leq 
{1 \over |\mu - \tau|^2 } + { 1\over |\lambda -\tau|^2 } ,$$
yields the estimate \eqref{eq:Estim-1}.
\end{proof}

The next lemma takes care of the quadratic terms $\langle \psi_{m,k}, e_{\zeta_{0}}\rangle\, \langle e_{*\zeta_{1}},\psi_{m,k}\rangle$ appearing in \eqref{eq:Diff-F1-F2}:

\begin{lemma}\label{lem:Estim-2}
For $\psi\in L^2(\Gamma)$ define
$$g(\psi) := \langle \psi, e_{\zeta_{0}}\rangle\, \langle e_{*\zeta_{1}},\psi\rangle.$$
Then there exists a constant $c > 0$ such that, provided that ${\rm Im}(\lambda) \geq 1$ and the vectors $\zeta_{0}$ and $\zeta_{1}$ are as in Lemmas \ref{lem:Larger-than-zero}--\ref{lem:Equal-to-zero},  for all $\psi_{1},\psi_{2} \in L^2(\Gamma)$ we have
\begin{equation}\label{eq:Estim-2}
|g(\psi_{1}) - g(\psi_{2}) | \leq c\, \left( |\langle e_{*\zeta_{1}},\psi_{1}\rangle | + 
| \langle \psi_{2},e_{\zeta_{0}}\rangle|\right)\, \|\psi_{1} - \psi_{2}\|_{L^2(\Gamma)}.
\end{equation}
\end{lemma}
Indeed, according to \eqref{eq:Behaviour-zeta} of Lemma \ref{lem:Behaviour-zeta} we have  $\|e_{\zeta_{j}}\|_{L^2(\Gamma)} \leq c $, and since
$$g(\psi_{1}) - g(\psi_{2}) = \langle \psi_{1} - \psi_{2},e_{\zeta_{0}} \rangle \langle e_{*\zeta_{1}},\psi_{1} \rangle +
\langle \psi_{2},e_{\zeta_{0}} \rangle \langle e_{*\zeta_{1}},\psi_{1} - \psi_{2} \rangle, $$
one sees that \eqref{eq:Estim-2} follows. \qed

\medskip
We can now state the following result regarding the limit in \eqref{eq:Diff-F1-F2} as the parameter $\mu \to -\infty$:

\begin{lemma}\label{lem:Lim-mu-F1-F2}
Assume that $\lambda,\zeta_{0},\zeta_{1}$ are as in Theorem \ref{thm:Isozaki} and that ${\rm Im}(\lambda) \geq 1$. Moreover assume that we have
$$\sum_{k\geq 1}\|\psi_{1,k} - \psi_{2,k}\|_{L^2(\Gamma)}^2  < \infty.$$
Then 
$$\sum_{k \in {\Bbb K}_{0}}
\left| 
{\langle \psi_{1,k}, e_{\zeta_{0}}\rangle\, \langle e_{*\zeta_{1}},\psi_{1,k} \rangle  \over \lambda - \lambda_{1,k}}
- {\langle \psi_{2,k},e_{\zeta_{0}}\rangle\, \langle e_{*\zeta_{1}},\psi_{2,k} \rangle  \over \lambda - \lambda_{2,k}}\right| < \infty$$
and we have 
\begin{eqnarray}
&\displaystyle\lim_{\mu\to -\infty}\langle e_{*\zeta_{1}}, F_{1}(\lambda,\mu,e_{\zeta_{0}}) - F_{2}(\lambda,\mu,e_{\zeta_{0}}) \rangle &= \label{eq:Lim-mu-F1-F2}\\
&& \hskip-4cm\sum_{k \in {\Bbb K}_{0}}
\left[ 
{\langle \psi_{1,k}, e_{\zeta_{0}}\rangle\, \langle e_{*\zeta_{1}},\psi_{1,k} \rangle  \over \lambda_{1,k} - \lambda}
- {\langle \psi_{2,k},e_{\zeta_{0}}\rangle\, \langle e_{*\zeta_{1}},\psi_{2,k} \rangle  \over \lambda_{2,k} - \lambda}\right]. \nonumber
\end{eqnarray}
\end{lemma}

\begin{proof}
For $\lambda$ fixed, we may write each term of the series appearing in \eqref{eq:Diff-F1-F2} as
\begin{equation}\label{eq:Series-F1-F2}
(\mu - \lambda)\left[ {\langle \psi_{1,k}, e_{\zeta_{0}}\rangle\, \langle e_{*\zeta_{1}},\psi_{1,k} \rangle  \over (\lambda - \lambda_{1,k})(\mu - \lambda_{1,k})}
- {\langle \psi_{2,k},e_{\zeta_{0}}\rangle\, \langle e_{*\zeta_{1}},\psi_{2,k} \rangle  \over (\lambda - \lambda_{2,k})(\mu - \lambda_{2,k})}\right] = A_{k}(\mu) + B_{k}(\mu)
\end{equation}
where for convenience, with the notations of Lemmas \ref{lem:Estim-1} and \ref{lem:Estim-2}, we have set
\begin{equation}\label{eq:Def-A-mu}
A_{k}(\mu) := \left(f_{\lambda,\mu}(\lambda_{1,k}) - f_{\lambda,\mu}(\lambda_{2,k})\right)g(\psi_{1,k}),
\end{equation}
and
\begin{equation}\label{eq:Def-B-mu}
B_{k}(\mu) := f_{\lambda,\mu}(\lambda_{2,k})\left(g(\psi_{1,k}) - g(\psi_{2,k})\right).
\end{equation}
Setting
\begin{equation}\label{eq:Def-A*}
A_{*k}(\lambda) := \left({1 \over \lambda_{1,k} - \lambda} - {1 \over \lambda_{2,k} - \lambda}\right)g(\psi_{1,k}),
\end{equation}
and 
\begin{equation}\label{eq:Def-B*}
B_{*k}(\lambda) := {1 \over \lambda_{2,k} - \lambda}\left(g(\psi_{1,k}) - g(\psi_{2,k})\right),
\end{equation}
it is clear that for each fixed $k \geq 1$ we have
$$
\lim_{\mu\to -\infty}A_{k}(\mu) = A_{*k}(\lambda),\qquad\mbox{and}\qquad
\lim_{\mu\to -\infty}B_{k}(\mu) = B_{*k}(\lambda).
$$
It is also clear that $A_{*k}(\lambda) + B_{*k}(\lambda)$ is precisely the generic term of the series appearing on the right hand side of \eqref{eq:Lim-mu-F1-F2}. Therefore we have only to justify the passage to the limit in the series as $\mu\to-\infty$, while $\lambda$ is fixed.

Denote by $\lambda_{k}^*$ a number between $\lambda_{1,k}$ and $\lambda_{2,k}$ where the maximum of
$$\left[{1 \over |\lambda -\tau|^2} + {1 \over |\mu - \tau|^2}\right]$$
is achieved as $\tau$ ranges between $\lambda_{1,k}$ and $\lambda_{2,k}$. 
Recall that thanks to \eqref{eq:Estim-ecart-vp}, when $t = \mu$ or $t=\lambda$ we have $|t - \lambda_{m,k}| \geq 1$, and thus
\begin{equation}\label{eq:Rapport-12}
\left|{t - \lambda_{2,k} \over t - \lambda_{1,k}}\right| =\left| 1 + {\lambda_{1,k} - \lambda_{2,k} \over t - \lambda_{1,k}}\right| \leq 1 + \|V_{1} - V_{2}\|_{\infty} \leq c,
\end{equation}
for some constant $c$ independent of $k$ and $t$. Indeed, for another constant $c>0$, and $m=1$ or $m=2$, for all $k$, one can see that we have also
$$\left|{t - \lambda_{k}^* \over t - \lambda_{m,k}}\right| 
+ \left|{t - \lambda_{m,k} \over t - \lambda_{k}^*}\right|\leq c.$$
Therefore, by \eqref{eq:Series-u-lambda}, for $\zeta = \zeta_{0}$ or $\zeta=\zeta_{1}$ and, when necessary, $\lambda$ replaced with $\mu$, we have
\begin{equation}\label{eq:Series-u-lambda-bis}
\sum_{k \geq 1}\left|{\langle \psi_{m,k},e_{\zeta}\rangle \over \lambda - \lambda_{k}^*}\right|^2 < \infty,\quad\mbox{and}\quad
\sum_{k \geq 1}\left|{\langle \psi_{m,k},e_{\zeta}\rangle \over \mu - \lambda_{k}^*}\right|^2 < \infty.
\end{equation}
From this, using Lemma \ref{lem:Estim-1} we infer that, setting $\delta_{0}:= \sup_{k\geq 1}|\lambda_{1,k} - \lambda_{2,k}|$,
$$ |A_{k}(\mu)| \leq 2\,\delta_{0} \,\left[ 
 {|g(\psi_{1,k})| \over |\mu - \lambda_{k}^*|^2} + 
{|g(\psi_{1,k})| \over |\lambda - \lambda_{k}^*|^2}\right].$$
Now, for $\mu \leq - (M+1)$, we have $\lambda_{k}^* - \mu \geq \lambda_{k}^* + M + 1$ where $M > 0$ is such that $\lambda_{m,k} \geq -M$ for all $k \geq 1$ and $m=1$ or $m=2$. It follows that for all $\mu \leq -(M+1)$ we have
$$
|A_{k}(\mu)| \leq 2\,\delta_{0} \,\left[ {|g(\psi_{1,k})| \over |\lambda_{k}^* + M +1|^2} + 
{|g(\psi_{1,k})| \over |\lambda - \lambda_{k}^*|^2}\right],
$$
where the right hand side is summable over $k$.
It is now clear that we may apply Lebesgue's dominated convergence theorem and deduce that
\begin{equation}\label{eq:Lim-A-mu}
\lim_{\mu\to-\infty}\sum_{k\in {\Bbb K}_{0}} A_{k}(\mu) = \sum_{k\in {\Bbb K}_{0}} \left({1 \over \lambda_{1,k} - \lambda} - {1 \over \lambda_{2,k} - \lambda}\right)g(\psi_{1,k}),
\end{equation}
and moreover we have
\begin{equation}\label{eq:Estim-sum-Ak}
\lim_{\mu\to-\infty}\sum_{k \in {\Bbb K}_{0}}|A_{k}(\mu)| \leq 2\,\delta_{0} \sum_{k \in {\Bbb K}_{0}}
{|g(\psi_{1,k})| \over |\lambda - \lambda_{k}^*|^2}.
\end{equation}

Regarding the terms $B_{k}(\mu)$ defined in \eqref{eq:Def-B-mu}, we proceed analogously: using Lemma \ref{lem:Estim-2} we have
$$|B_{k}(\mu)| \leq c\, |\lambda - \mu|  \,
{|\langle\psi_{2,k},e_{\zeta_{0}}\rangle| + |\langle e_{*\zeta_{1}},\psi_{1,k}\rangle| 
\over |\lambda - \lambda_{2,k}|\cdot |\mu - \lambda_{2,k}|}\, \|\psi_{1,k} - \psi_{2,k}\|_{L^2(\Gamma)}.$$
Again, as $\mu \to - \infty$ and $\lambda$ is fixed, we may find a positive constant $c_{1}$ such that $|\lambda - \mu| / |\mu - \lambda_{2,k}| \leq c_{1}$, so that finally (for another constant $c_{2} > 0$) we may write
$$|B_{k}(\mu)| \leq c_{2}\,
{|\langle\psi_{2,k},e_{\zeta_{0}}\rangle| + |\langle e_{*\zeta_{1}},\psi_{1,k}\rangle| 
\over |\lambda - \lambda_{2,k}|}\, \|\psi_{1,k} - \psi_{2,k}\|_{L^2(\Gamma)}.$$
Thanks to the assumption that
$$\delta_{1} := \left(\sum_{k\geq 1}\|\psi_{1,k} - \psi_{2,k}\|_{L^2(\Gamma)}^2\right)^{1/2} < \infty,$$
and the observation \eqref{eq:Series-u-lambda-bis}, we infer (using Young's inequality $2ab \leq a^2 + b^2$) that the right hand side of the above inequality is summable and applying once more Lebesgue's dominated convergence theorem we have
\begin{equation}\label{eq:Estim-sum-Bk}
\lim_{\mu\to -\infty}\sum_{k \geq N} |B_{k}(\mu)| \leq 2\,c_{2}\, \left( \sum_{k \geq N}
{|\langle\psi_{2,k},e_{\zeta_{0}}\rangle|^2 + |\langle e_{*\zeta_{1}},\psi_{1,k}\rangle|^2 
\over |\lambda - \lambda_{2,k}|^2}\right)^{1/2} \delta_{1},
\end{equation}
and 
$$\lim_{\mu\to-\infty}\sum_{k\in {\Bbb K}_{0}} B_{k}(\mu) = \sum_{k\in {\Bbb K}_{0}}
{1 \over \lambda_{2,k} - \lambda}\left(g(\psi_{1,k}) - g(\psi_{2,k})\right) .$$
At this point one sees that this, together with \eqref{eq:Lim-A-mu}, yield \eqref{eq:Lim-mu-F1-F2}, and the proof of the lemma is complete.
\end{proof}

\medskip
Next, $A_{*k}(\lambda)$ and $B_{*k}(\lambda)$ being defined by \eqref{eq:Def-A*} and \eqref{eq:Def-B*} we study the series 
$$\sum_{k\in {\Bbb K}_{0}}A_{*k}(\lambda), \qquad\mbox{and}\quad
\sum_{k\in {\Bbb K}_{0}}B_{*k}(\lambda).$$

\begin{lemma}\label{lem:Lim-lambda-A*k}
With the above assumptions on $\zeta_{0},\zeta_{1}$ and $V_{1},V_{2}$, assume that  ${\rm Im}(\lambda) \geq 1$ and let $A_{*k}$ be defined by \eqref{eq:Def-A*}. Then there exists a positive constant $c$ depending only on $\omega$ and $M$ such that for all integer $N \geq 1$ we we have 
$$\limsup_{|\lambda|\to\infty}\sum_{k \in {\Bbb K}_{0}}
\left|A_{*k}(\lambda)\right| \leq c\, \sup_{k \geq N}|\lambda_{1,k} - \lambda_{2,k}|.$$
\end{lemma}

\begin{proof}
Indeed, setting $\delta_{0} := \delta_{0}(N) := \sup_{k \geq N}|\lambda_{1,k} - \lambda_{2,k}|$, using the very definition of $A_{*k}(\lambda)$, and recalling \eqref{eq:Estim-ecart-vp} and \eqref{eq:Rapport-12}, we observe that
\begin{eqnarray*}
& \displaystyle \sum_{k \geq N}|A_{*k}(\lambda)| &\leq \delta_{0}(N) \,
	\sum_{k \geq N} {|\langle \psi_{1,k},e_{\zeta_{0}}\rangle| \over |\lambda - \lambda_{1,k}|}
	\cdot 
	{|\langle e_{\zeta_{*1}},\psi_{1,k}\rangle| \over |\lambda - \lambda_{2,k}|}\\
	&& \leq {\delta_{0}(N) \over 2} \left(\sum_{k \geq N} {|\langle \psi_{1,k},e_{\zeta_{0}}\rangle|^2 \over |\lambda - \lambda_{1,k}|^2} + 
	\sum_{k \geq N} {|\langle e_{*\zeta_{1}},\psi_{1,k} \rangle|^2 \over |\lambda - \lambda_{2,k}|^2}\right)\\
	&& \leq {\delta_{0}(N) \over 2} \left(\sum_{k \geq N} 
	{|\langle \psi_{1,k},e_{\zeta_{0}}\rangle|^2 \over 
		|\lambda - \lambda_{1,k}|^2} + 
	c \, \sum_{k \geq N} 
	{|\langle e_{*\zeta_{1}},\psi_{1,k} \rangle|^2 \over 
		|\lambda - \lambda_{1,k}|^2}\right) \\
	&& \leq c\, \delta_{0}(N),
\end{eqnarray*}
for some positive constant $c$ independent of $\lambda$, where in the last step we use the fact that if $u$ solves equation \eqref{eq:u-e-zeta} with $\zeta=\zeta_{0}$, and $u_{*}$ solves \eqref{eq:u*-e-zeta} (both equations with $V=V_{1}$), by \eqref{eq:Series-u-lambda} and \eqref{eq:Estim-u-e-zeta}, we have
$$\sum_{k \geq 1} {|\langle \psi_{1,k},e_{\zeta_{0}}\rangle|^2 \over |\lambda - \lambda_{1,k}|^2} 
+ \sum_{k \geq 1} {|\langle e_{*\zeta_{1},\psi_{1,k}}\rangle|^2 \over |\lambda - \lambda_{1,k}|^2}
= \|u\|_{L^2(Y)}^2 + \|u_{*}\|_{L^2(Y)}^2 \leq c.$$
On the other hand, for $N \geq 1$ fixed, it is clear that for some positive constant $c_{1}$ depending on $N$ (but not on $\lambda$) we have
$$\forall\, k \in [1,N]\cap {\Bbb N},\qquad |A_{*k}(\lambda)| \leq {c_{1} \over |\lambda -\lambda_{1,k}|\cdot|\lambda - \lambda_{2,k}|},$$
so that 
$$\lim_{|\lambda|\to \infty}\sum_{k=1}^{N} |A_{*k}(\lambda)| = 0.$$
It follows that
$$\limsup_{|\lambda|\to\infty} \sum_{k\in {\Bbb K}_{0}}|A_{*k}(\lambda)| \leq c\, \delta_{0}(N),$$
and the proof of our lemma is complete.
\end{proof}

Now we turn our attention to the series defined by $B_{*k}(\lambda)$. 

\begin{lemma}\label{lem:Lim-lambda-B*k}
With the above assumptions on $\zeta_{0},\zeta_{1}$ and $V_{1},V_{2}$, assume that  ${\rm Im}(\lambda) \geq 1$ and let $B_{*k}$ be defined by \eqref{eq:Def-B*}.
Moreover assume that 
$$\sum_{k\geq 1}\|\psi_{1,k} - \psi_{2,k}\|_{L^2(\Gamma)}^2  < \infty.$$
Then 
\begin{equation}\label{eq:Limsup-sum-B*k}
\limsup_{|\lambda|\to\infty}\sum_{k \in {\Bbb K}_{0}}
\left|B_{*k}(\lambda)\right| = 0.
\end{equation}
\end{lemma}

\begin{proof}
For a given $\epsilon > 0$, we may fix $N_{\epsilon} \geq 1$ such that
$$\sum_{k \geq N_{\epsilon} +1}\|\psi_{1,k} - \psi_{2,k}\|_{L^2(\Gamma)}^2 \leq \epsilon^2. $$
However, by Lemma \ref{lem:Estim-2} we have
\begin{eqnarray}
& |B_{*k}(\lambda)| &= {1 \over |\lambda_{2,k} - \lambda|}\, \left|g(\psi_{1,k}) - g(\psi_{2,k})\right| \nonumber\\
&& \leq\quad {|\langle e_{*\zeta_{1}},\psi_{1,k}\rangle | + 
| \langle \psi_{2,k},e_{\zeta_{0}}\rangle| \over |\lambda_{2,k} - \lambda|} \, \|\psi_{1,k} - \psi_{2,k}\|_{L^2(\Gamma)}.\label{eq:Estim-B*k-1-0}
\end{eqnarray}
Using the fact that 
$${\left(|\langle e_{*\zeta_{1}},\psi_{1,k}\rangle | + 
| \langle \psi_{2,k},e_{\zeta_{0}}\rangle|\right)^2 \over |\lambda_{2,k} - \lambda|^2} \leq
c \left( {|\langle e_{*\zeta_{1}},\psi_{1,k}\rangle |^2
\over |\lambda_{1,k} - \lambda|^2} + 
{|\langle \psi_{2,k},e_{\zeta_{0}}\rangle |^2
\over |\lambda_{2,k} - \lambda|^2}\right),$$
invoking once more Lemma \ref{lem:Estim-u-e-zeta} and \eqref{eq:Estim-u-e-zeta}, we infer that 
$$\sum_{k \geq 1} {\left(|\langle e_{*\zeta_{1}},\psi_{1,k}\rangle | + 
| \langle \psi_{2,k},e_{\zeta_{0}}\rangle|\right)^2 \over |\lambda_{2,k} - \lambda|^2} \leq c$$
for some positive constant independent of $\lambda$. Consequently, using \eqref{eq:Estim-B*k-1-0} and the Cauchy-Schwarz' inequality in $\ell^2$, we have
$$\sum_{k \geq N_{\epsilon}+1}|B_{*k}{\lambda}| \leq c\, \left(\sum_{k \geq N_{\epsilon} +1}\|\psi_{1,k} - \psi_{2,k}\|_{L^2(\Gamma)}^2\right)^{1/2} \leq c\, \epsilon.$$
On the other hand, as we argued above in the study of $\sum_{k=1}^N |A_{*k}(\lambda)|$, it is not difficult to see that
$$\lim_{|\lambda|\to\infty}\sum_{k=1}^{N_{\epsilon}} |B_{*k}(\lambda)| = 0,$$
so that finally we get
$$\limsup_{|\lambda|\to\infty} \sum_{k\in {\Bbb K}_{0}}|B_{*k}(\lambda)| \leq c\, \epsilon,$$
and $\epsilon >0$ being arbitrary, this yields
$\limsup_{|\lambda|\to\infty} \sum_{k\in {\Bbb K}_{0}}|B_{*k}(\lambda)| = 0$, so that the proof of the lemma is complete.
\end{proof}
\medskip

Now we are in a position to develop the proof of our main results. We begin with that of Theorem \ref{thm:Uniqueness-QP}.


\subsection{Proof of Theorems \ref{thm:Uniqueness-QP} and \ref{thm:Uniqueness}}\label{subsec:Dem-thm:Uniqueness-QP}

It is clear that Theorem \ref{thm:Uniqueness} is nothing but a rewriting of Theorem \ref{thm:Uniqueness-QP} in the context of an infinite waveguide. We can therefore develop only the proof of Theorem \ref{thm:Uniqueness-QP}. 

According to the assumptions of this theorem, the set ${\Bbb K}_{0}$ defined in \eqref{eq:Def-K-0} is finite (in fact ${\rm card}({\Bbb K}_{0}) \leq N$).
Thus, using \eqref{eq:Id-lambda}, the sum in \eqref{eq:Def-F-*} being finite in this case, we may take the limit as $|\lambda|\to \infty$ to obtain
$$\lim_{|\lambda|\to\infty} \left(S_{\theta,V_{1}}(\lambda,\zeta_{0},\zeta_{1}) - S_{\theta,V_{2}}(\lambda,\zeta_{0},\zeta_{1})\right) = 0,
$$
so that \eqref{eq:V-hat-and-S1-S2} implies that $V_{1} \equiv V_{2}$, and the proof of Theorem \ref{thm:Uniqueness-QP} is complete. \qed


\subsection{Proof of Theorem \ref{thm:Stability}}\label{subsec:Dem-thm:Stability}

To begin with, let us recall that if we assume that \eqref{eq:Ecart-ell-2-psi} is satisfied and moreover
$$\lim_{k\to\infty}|\lambda_{1,k} - \lambda_{2,k}| \to 0,$$
then as $N \to \infty$ we have have $\sup_{k\geq N}|\lambda_{1,k} - \lambda_{2,k}| \to 0$, and therefore  Theorem \ref{thm:Uniqueness-Asymptotic} is an easy consequence of Theorem \ref{thm:Stability}. We can therefore concentrate on the proof of this latter result.

\medskip

In order to explain the various steps we are going to take for the proof of Theorem \ref{thm:Stability}, we begin by noting that if $V := (V_{1} - V_{2})1_{Y}$, and if we set for a moment $y := (\xi',-j)$ and $z := {\rm i}(\zeta_{0} + \zeta_{1})$, then we have
$$|\widehat{V}(y) - \widehat{V}(z)| \leq |y - z| \, \sup_{0 \leq \tau \leq 1}|\nabla \widehat{V}((1-\tau)y + \tau z)|. $$
Now recall that by Lemma \ref{lem:Behaviour-zeta} we have $|y - z| = O(|\lambda|^{-1/2})$ as $|\lambda| \to +\infty$, and thus, if $R > 1$ is given, for $|\xi'|^2 + j^2 \leq R^2$, we have 
$$\sup_{0 \leq \tau \leq 1}|\nabla \widehat{V}((1-\tau)y + \tau z)| \leq 2R\, \|V\|_{L^1(Y)} \leq 8R\pi\, {\rm meas}(\omega)\, M , $$
we conclude that for some constant $c_{1} > 0$ depending on $\omega$, $M$ and $R$ we have
\begin{equation}\label{eq:Estim-V-hat}
|\widehat{V}(\xi',-j)| \leq |\widehat{V}({\rm i}(\zeta_{0}+\zeta_{1}))| + c_{1}\,|\lambda|^{-1/2}.
\end{equation}
Therefore, in order to prove the estimate \eqref{eq:Estim-stability} we need to analyze the behaviour of $|\widehat{V}({\rm i}(\zeta_{0}+\zeta_{1}))|$ as $|\lambda|\to \infty$. Now, thanks to the representation formula \eqref{eq:Id-Poinca-Steklov}, we have
\begin{equation}\label{eq:Estim-V-hat-1}
|\widehat{V}({\rm i}(\zeta_{0}+\zeta_{1}))| \leq |S_{\theta,V_{1}}(\lambda,e_{\zeta_{0}},e_{\zeta_{1}}) - S_{\theta,V_{2}}(\lambda,e_{\zeta_{0}},e_{\zeta_{1}})| + c_{2}\,|\lambda|^{-1/2},
\end{equation}
where we have used on the one hand the fact that when $\|V_{m}\|_{\infty} \leq M$, then for some positive constant $c_{3}$
$$\left|\int_{Y}R_{\lambda}(V_{m}e_{\zeta_{0}})(x)V_{m}(x) e_{\zeta_{1}}(x)\,dx \right| \leq c_{3}\, \|(A_{m} - \lambda I)^{-1}\|_{L^2 \to L^2} = O(|{\rm Im}(\lambda)|^{-1}) ,$$
and on the other hand the fact that by Lemma \ref{lem:Behaviour-zeta} we have $O(|{\rm Im}(\lambda)|^{-1}) = O(|\lambda|^{-1/2})$.

Reporting \eqref{eq:Estim-V-hat-1} into \eqref{eq:Estim-V-hat}, we see that
\begin{equation}\label{eq:Estim-V-hat-2}
|\widehat{V}(\xi',-j)| \leq |S_{\theta,V_{1}}(\lambda,e_{\zeta_{0}},e_{\zeta_{1}}) - S_{\theta,V_{2}}(\lambda,e_{\zeta_{0}},e_{\zeta_{1}})| + c_{4}\,|\lambda|^{-1/2},
\end{equation}
and again we have to analyze the behaviour of $|S_{\theta,V_{1}}(\lambda,e_{\zeta_{0}},e_{\zeta_{1}}) - S_{\theta,V_{2}}(\lambda,e_{\zeta_{0}},e_{\zeta_{1}})|$ as $|\lambda|\to +\infty$.

According to \eqref{eq:Id-lambda-mu}, with the choice $f:= e_{\zeta_{0}}$ we have
\begin{equation*}
\left\langle e_{*\zeta_{1}},{\partial u_{1,\lambda} \over \partial{\bf n}} - 
{\partial u_{2,\lambda} \over \partial{\bf n}}\right\rangle = 
\left\langle e_{*\zeta_{1}},{\partial z_{\mu} \over \partial{\bf n}} \right\rangle
+ \langle e_{*\zeta_{1}}, F_{1}(\lambda,\mu,e_{\zeta_{0}}) - F_{2}(\lambda,\mu,e_{\zeta_{0}}) \rangle,
\end{equation*}
and by Definition \ref{def:S-V} of $S_{\theta,V_{m}}(\lambda,e_{\zeta_{0}},e_{\zeta_{1}})$
\begin{equation*}
\left\langle e_{*\zeta_{1}},{\partial u_{1,\lambda} \over \partial{\bf n}} - 
{\partial u_{2,\lambda} \over \partial{\bf n}}\right\rangle =
S_{\theta,V_{1}}(\lambda,e_{\zeta_{0}},e_{\zeta_{1}}) - S_{\theta,V_{2}}(\lambda,e_{\zeta_{0}},e_{\zeta_{1}}).
\end{equation*}
Therefore, in order to deduce the result of Theorem \ref{thm:Stability} from \eqref{eq:Estim-V-hat-2}, we need to analyze the difference given by \eqref{eq:Diff-F1-F2}
as $\mu\to -\infty$ and $|\lambda|\to + \infty$. 

First, using Lemma \ref{lem:Lim-mu-F1-F2} together with \eqref{eq:Estim-V-hat-2} we may state the following

\begin{lemma}\label{lem:Lim-mu-F1-F2-bis}
With our assumption on $M,\lambda,\zeta_{0},\zeta_{1}$, as well as on $V_{1},V_{2}$,  assume that  ${\rm Im}(\lambda) \geq 1$ and that we have
$$\sum_{k\geq 1}\|\psi_{1,k} - \psi_{2,k}\|_{L^2(\Gamma)}^2 < \infty.$$
Then 
$$\sum_{k \in {\Bbb K}_{0}}
\left| 
{\langle \psi_{1,k}, e_{\zeta_{0}}\rangle\, \langle e_{*\zeta_{1}},\psi_{1,k} \rangle  \over \lambda - \lambda_{1,k}}
- {\langle \psi_{2,k},e_{\zeta_{0}}\rangle\, \langle e_{*\zeta_{1}},\psi_{2,k} \rangle  \over \lambda - \lambda_{2,k}}\right| < \infty\, ,$$
and for a constant $c_{*}$ depending on $\omega$, $M$, and $R$, for all $(\xi',j)$ such that $|\xi'|^2 + j^2 \leq R^2$ we have
\begin{equation}\label{eq:Estim-V-hat-3} 
|\widehat{V}(\xi',-j)| \leq \sum_{k\in {\Bbb K}_{0}}\left| 
{\langle \psi_{1,k}, e_{\zeta_{0}}\rangle\, \langle e_{*\zeta_{1}},\psi_{1,k} \rangle  \over \lambda_{1,k} - \lambda} 
- {\langle \psi_{2,k},e_{\zeta_{0}}\rangle\, \langle e_{*\zeta_{1}},\psi_{2,k} \rangle  \over \lambda_{2,k} - \lambda}\right| + c_{*}\,|\lambda|^{-1/2}.
\end{equation}

\end{lemma}

\smallskip
Now we are in a position to prove Theorem \ref{thm:Stability}:
\medskip

\noindent{\bf Proof of Theorem \ref{thm:Stability} concluded}

\smallskip
\noindent With the notations introduced in \eqref{eq:Def-A*} and \eqref{eq:Def-B*}, the estimate \eqref{eq:Estim-V-hat-3} may be rewritten as
\begin{eqnarray}
&|\widehat{V}(\xi',-j)| &\leq \sum_{k\in {\Bbb K}_{0}}\left| 
A_{*k}(\lambda) + B_{*k}(\lambda)\right| + c\, |\lambda|^{-1/2} \nonumber\\
&& \leq \sum_{k\in {\Bbb K}_{0}}\left| 
A_{*k}(\lambda)\right| + \left| B_{*k}(\lambda)\right| + c\,|\lambda|^{-1/2}. \label{eq:Estim-V-hat-4}
\end{eqnarray}
Now, thanks to Lemma \ref{lem:Lim-lambda-B*k} we know that
$$\lim_{|\lambda|\to\infty} \sum_{k\in{\Bbb K}_{0}} |B_{*k}(\lambda)| = 0,$$
and by Lemma \ref{lem:Lim-lambda-A*k} we have, with the constant $c > 0$ depending only on $\omega$ and $M$, 
$$\limsup_{|\lambda|\to\infty} \sum_{k\in{\Bbb K}_{0}} |A_{*k}(\lambda)| \leq c\, \sup_{k \geq N}|\lambda_{1,k} - \lambda_{2,k}|. $$
Therefore $|\widehat{V}(\xi',-j)|\leq c\, \sup_{k \geq N}|\lambda_{1,k} - \lambda_{2,k}|$ as claimed.
\qed


\bigskip

\end{document}